\documentclass[twoside,a4paper,reqno,11pt]{amsart} 
\usepackage[top=28mm,right=28mm,bottom=28mm,left=28mm]{geometry}

\usepackage{amsfonts, amsmath, amssymb, mathrsfs, bm, latexsym, stmaryrd, array, hyperref, mathtools, bold-extra}

\renewcommand{\a}{\alpha}
\renewcommand{\b}{\beta}

\newcommand{\e}{\varepsilon}
\renewcommand{\l}{\lambda} 

\newcommand{\s}{\sigma}

\renewcommand{\O}{\Omega}

\newcommand{\normeq}{\trianglelefteqslant}

\newcommand{\la}{\langle}
\newcommand{\ra}{\rangle}

\renewcommand{\to}{\rightarrow}

\newcommand{\leqs}{\leqslant}
\newcommand{\geqs}{\geqslant}

\newcommand{\vs}{\vspace{2mm}}

\newcommand{\what}{\widehat}

\makeatletter
\newcommand{\imod}[1]{\allowbreak\mkern4mu({\operator@font mod}\,\,#1)}
\makeatother

\newtheorem{theorem}{Theorem} 
\newtheorem*{conj*}{Conjecture}

\newtheorem{corol}[theorem]{Corollary}

\newtheorem{thm}{Theorem}[section] 
\newtheorem{prop}[thm]{Proposition} 
\newtheorem{lem}[thm]{Lemma}
\newtheorem{cor}[thm]{Corollary}

\theoremstyle{definition}
\newtheorem{rem}[thm]{Remark}
\newtheorem{remk}{Remark}

\newtheorem{defn}[thm]{Definition}

\begin{document}

\author{Timothy C. Burness}
\thanks{}
\address{T.C. Burness, School of Mathematics, University of Bristol, Bristol BS8 1UG, UK}
\email{t.burness@bristol.ac.uk}
 
  \author{Adam R. Thomas}
 \address{A.R. Thomas, Mathematics Institute,
Zeeman Building, University of Warwick, Coventry CV4 7AL, UK}
 \email{adam.r.thomas@warwick.ac.uk}

\title[Normalisers of maximal tori]{Normalisers of maximal tori and \\ a conjecture of Vdovin}

\begin{abstract}
Let $G = O^{p'}(\bar{G}^F)$ be a finite simple group of Lie type defined over a field of characteristic $p$, where $F$ is a Steinberg endomorphism of the ambient simple algebraic group $\bar{G}$. Let $\bar{T}$ be an $F$-stable maximal torus of $\bar{G}$ and set $N = N_G(\bar{T})$. A conjecture due to Vdovin asserts that if $G \not\cong {\rm L}_3(2)$ then $N \cap N^x$ is a $p$-group for some $x \in G$. In this paper, we use a combination of probabilistic and computational methods to calculate the base size for the natural action of $G$ on $G/N$, which allows us to prove a stronger, and suitably modified, version of Vdovin's conjecture. 
\end{abstract}

\date{\today}

\maketitle

\section{Introduction}\label{s:intro}

Let $G$ be a finite simple group of Lie type over $\mathbb{F}_q$, where $q=p^f$ and $p$ is a prime. Write $G = O^{p'}(\bar{G}^F)$, where $\bar{G}$ is a simple algebraic group of adjoint type over the algebraic closure $k$ of $\mathbb{F}_p$ and $F$ is an appropriate Steinberg endomorphism of $\bar{G}$ with fixed point subgroup $\bar{G}^F$. Let $\bar{T}$ be an $F$-stable maximal torus of $\bar{G}$ and set $N = N_G(\bar{T})$. The following conjecture of Evgeny Vdovin is presented as Problem 17.42 in the Kourovka Notebook \cite{Kou} (it first appeared in the 17th edition, which was published in 2010).

\begin{conj*}
Let $G = O^{p'}(\bar{G}^F)$ be a finite simple group of Lie type and define $N = N_{G}(\bar{T})$ as above. If $G \not\cong {\rm L}_3(2)$ then $N \cap N^x$ is a $p$-group for some $x \in G$.
\end{conj*}

The special case $G \cong {\rm L}_{3}(2)$ is a genuine exception. Indeed, if $N = 7{:}3$ is the normaliser of a Singer cycle, then $|N \cap N^x| = 3$ for all $x \in G \setminus N$. 

In this paper, we will prove a suitably modified version of Vdovin's conjecture (it turns out that there are two additional exceptions, so an adjustment is necessary). In order to state our main result (see Theorem \ref{t:main} below), we need some additional terminology. Let $G$ and $N$ be as above and view $G$ as a transitive permutation group on the set $\O = G/N$ of cosets of $N$. Then the \emph{base size} of $G$, denoted $b(G,N)$, is the minimal size of a subset of $\O$ with trivial pointwise stabiliser in $G$. Equivalently, $b(G,N)$ is the smallest number of conjugates of $N$ so that the intersection of these subgroups is trivial. In particular, $b(G,N)=2$ if and only if $N \cap N^x = 1$ for some $x \in G$. As a consequence, let us observe that if $p$ does not divide the order of the Weyl group $N_{\bar{G}}(\bar{T})/\bar{T}$ of $\bar{G}$, then $N \cap N^x$ is a $p$-group for some $x \in G$ if and only if $b(G,N) = 2$.  

Determining the base size of a finite permutation group is both a classical and fundamental problem in permutation group theory (we refer the reader to the survey articles \cite{BC,LSh3} and \cite[Section 5]{Bur181} for more background on bases and their diverse applications in group theory and related areas). In particular, there has been a great deal of recent interest in studying base sizes for almost simple primitive permutation groups, partly motivated by a circle of highly influential conjectures of Babai, Cameron, Kantor and Pyber from the 1990s, which have all been resolved in recent years. In this context, our main result can be viewed as a contribution to research in this direction. It also constitutes further progress towards a classification of the finite primitive groups with a base of size $2$, which is an active and ambitious project initiated by Saxl in the 1990s.    

\begin{theorem}\label{t:main}
Let $G = O^{p'}(\bar{G}^F)$ be a finite simple group of Lie type and set $N = N_{G}(\bar{T})$ as above. Then either
\begin{itemize}\addtolength{\itemsep}{0.2\baselineskip}
\item[{\rm (i)}] $b(G,N) = 2$; or
\item[{\rm (ii)}] $G$, $N$ and $b(G,N)$ are recorded in Table \ref{tab:main} (up to isomorphism).
\end{itemize}
\end{theorem}

\begin{table}
\[
\begin{array}{clll} \hline
b(G,N) & G & N  & \\ \hline
3 & {\rm L}_{2}(q) & D_{2(q+1)} & \mbox{$q \geqs 4$ even} \\
& {\rm L}_{3}(2) & 7{:}3 & \\ 
& {\rm U}_{3}(3) & 4^2{:}S_3 & \\
& {\rm U}_5(2) & 3^4{:}S_5 & \\
& {\rm U}_6(2) & 3^4{:}S_6 & \\
& {\rm Sp}_6(2) & 3^3{:}(S_2 \wr S_3) & \\
& \O_{8}^{+}(2) & 3^4{:}(2^3{:}S_4) & \\ 
4 & {\rm U}_4(2) & 3^3{:}S_4 & \\ \hline
\end{array}
\]
\caption{The groups in Theorem \ref{t:main} with $b(G,N) \geqs 3$}
\label{tab:main}
\end{table}

By inspecting the cases appearing in Table \ref{tab:main}, we obtain the following corollary, which establishes a modified form of Vdovin's conjecture.

\begin{corol}\label{c:main}
There exists an element $x \in G$ such that $N \cap N^x$ is a $p$-group if and only if $(G,N)$ is not one of the following (up to isomorphism):
\begin{equation}\label{e:ex}
({\rm L}_{3}(2), 7{:}3), \; ({\rm U}_4(2),3^3{:}S_4), \; ({\rm U}_5(2),3^4{:}S_5).
\end{equation}
\end{corol}

\begin{remk}\label{r:main1}
Let us comment on the three special cases recorded in \eqref{e:ex}.
\begin{itemize}\addtolength{\itemsep}{0.2\baselineskip}
\item[{\rm (a)}] First assume $G = {\rm L}_{3}(2)$. Here $\bar{G} = {\rm PSL}_{3}(k)$ and we recall that there is a bijection between the set of $G$-classes of $F$-stable maximal tori of $\bar{G}$ and the set of conjugacy classes of the Weyl group $S_3$ (see Section \ref{ss:tori} for more details). One of these $G$-classes corresponds to the split maximal tori in $G$, which are trivial since $q=2$, so in this case we observe that $N = N_{G}(\bar{T}) = S_3$ is the subgroup of monomial matrices in $G$ and it is easy to check that $b(G,N) = 2$. The other two $G$-classes yield subgroups of the form $N_G(T)$ with $T$ a cyclic maximal torus of $G$. If $|T|=3$ then $N = D_6$ and $b(G,N) = 2$. On the other hand, if $|T|=7$ then $T$ is a Singer cycle, $N = 7{:}3$ and we find that $|N \cap N^x| = 3$ for all $x \in G \setminus N$ (in addition, it is easy to identify elements $x,y \in G$ such that $N \cap N^x \cap N^y = 1$, so $b(G,N) = 3$).
\item[{\rm (b)}] Next suppose $G = {\rm U}_4(2)$ and $N = 3^3{:}S_4$ is the normaliser of a split maximal torus. This case is also an exception to the main assertion in Vdovin's conjecture since $|N \cap N^x| \in \{24,54\}$ for all $x \in G \setminus N$. In fact, $b(G,N) = 4$ and it is worth noting that there exist $x,y \in G$ such that $N \cap N^x \cap N^y$ has order $2$.
\item[{\rm (c)}] Similarly, if $G = {\rm U}_5(2)$ and $N=3^4{:}S_5$ then we find that $b(G,N) = 3$ and $|N \cap N^x|$ is divisible by $12$ for all $x \in G\setminus N$. 
\end{itemize}
\end{remk}

Some special cases of Theorem \ref{t:main} have been studied in earlier work. For example, \cite[Proposition 4.2(i)]{BTh_ep} states that if $G$ is an almost simple exceptional group of Lie type and $N = N_G(\bar{T})$ is a maximal subgroup, then $b(G,N) = 2$. Similarly, if $G$ is an almost simple classical group and $N = N_G(\bar{T})$ is maximal, then $b(G,N) \leqs 4$ by the main theorem of \cite{B07} (in addition, the exact base size is computed in \cite{B20} if $N$ is soluble and maximal, which includes the cases studied here with $G = {\rm L}_2(q)$ and $q \geqs 13$).  

Bases for the action of the ambient algebraic group $\bar{G}$ on the coset variety $\bar{\O}=\bar{G}/\bar{N}$ have also been investigated, where $\bar{N} = N_{\bar{G}}(\bar{T})$. Here the main result is \cite[Theorem 9]{BGS}, which states that $b(\bar{G},\bar{N}) = 2$ unless $\bar{G}$ is isomorphic to ${\rm PSL}_2(k)$. (More precisely, if $\bar{G} \ne {\rm PSL}_2(k)$ then the \emph{generic} base size is $2$, which means that the $2$-point stabiliser $\bar{G}_{\a,\b}$ is trivial for all $(\a,\b)$ in a non-empty open subset of $\bar{\O} \times \bar{\O}$.) As a consequence, if $G \ne {\rm L}_2(q)$ is a finite simple group of Lie type over $\mathbb{F}_q$ then \cite[Proposition 2.7]{BGS} implies that $b(G,N) = 2$ for all sufficiently large $q$ (moreover, the probability $\mathcal{P}(G,N,2)$ that a random pair of points in $\O = G/N$ forms a base for $G$ tends to $1$ as $q$ tends to infinity). This result for algebraic groups is reflected in Theorem \ref{t:main}, where we see that Table \ref{tab:main} contains an infinite family of exceptions with $G = {\rm L}_2(q)$. In addition, let us observe that if $G = {\rm L}_2(q)$ then   
\[
\mathcal{P}(G,N,2) \to \left\{\begin{array}{ll} 1/2 & \mbox{if $q$ is odd} \\
0 & \mbox{otherwise}
\end{array}\right.
\]
as $q$ tends to infinity (see the proof of \cite[Lemmas 4.7, 4.8]{B20}). 

\begin{remk}\label{r:sol}
Our main theorem has already found an application in \cite{BLN}, which we briefly describe. Let $G$ be a finite insoluble group with souble radical $R(G)$ and consider the graph $\Gamma_{\mathcal{S}}(G)$ with vertices $G \setminus R(G)$, where distinct vertices $x$ and $y$ are adjacent if the subgroup $\la x,y \ra$ is soluble. This is called the \emph{soluble graph} of $G$ and the main theorem of \cite{BLN} states that $\Gamma_{\mathcal{S}}(G)$ is connected and its diameter, denoted $\delta_{\mathcal{S}}(G)$, is at most $5$. By a celebrated theorem of Thompson, a finite group is soluble if and only if every $2$-generated subgroup is soluble. This implies that $\delta_{\mathcal{S}}(G) \geqs 2$ and it remains an open problem to determine all the simple groups with $\delta_{\mathcal{S}}(G) = 2$; the only known examples are as follows (up to isomorphism):
\[
\mbox{${\rm L}_2(q)$ ($q \geqs 4$ even), ${\rm L}_3(2)$, ${\rm U}_4(2)$.}
\]
In \cite{BLN}, Theorem \ref{t:main} is used to reduce this problem to unitary, symplectic and orthogonal groups (see \cite[Propositions 6.9, 6.15]{BLN}). For groups of Lie type, the connection is as follows. Suppose there exists a semisimple element $g \in G$ such that $\la g \ra$ is a maximal torus and $N = N_G(\la g \ra)$ is the unique maximal soluble subgroup of $G$ containing $g$ (for example, if $n \geqs 3$ then this property holds if $G = {\rm L}_n(q)$ and $g \in G$ is a Singer element of order $(q^n-1)/d(q-1)$ with $d=(n,q-1)$). Excluding the special cases in Table \ref{tab:main}, it follows that $N \cap N^x=1$ for some $x \in G$. Since the neighbours of $g$ and $g^x$ in $\Gamma_{\mathcal{S}}(G)$ coincide with the nontrivial elements in $N$ and $N^x$, respectively, we conclude that the distance between $g$ and $g^x$ is at least $3$ and thus $\delta_{\mathcal{S}}(G) \geqs 3$. 
\end{remk}

We will apply a combination of probabilistic and computational methods in the proof of Theorem \ref{t:main}, handling the classical and exceptional groups separately. Our  main approach involves the application of fixed point ratio estimates in order to derive an upper bound on $\mathcal{Q}(G,N,2)$, which is the probability that a random pair of  points in $G/N$ do not form a base for $G$. Clearly, if $\mathcal{Q}(G,N,2)<1$ then $b(G,N) = 2$. This powerful method for studying base sizes was originally introduced by Liebeck and Shalev \cite{LSh} in their proof of a conjecture of Cameron and Kantor on bases for almost simple primitive groups. For appropriate low rank groups defined over small fields, we will also use a range of computational methods to calculate $b(G,N)$, working with {\sc Magma} \cite{Magma}. We refer the reader to Section \ref{ss:comp} for further details on some of these computations.

The classical groups require a detailed analysis and a key tool is the following zeta-type function
\[
\eta_G(t) = \sum_{i=1}^m|x_i^{G}|^{-t}
\]
where $t \in \mathbb{R}$ and $x_1, \ldots, x_m$ is a complete set of representatives of the conjugacy classes in $G$ of elements of prime order. In order to explain the relevance of this function, let 
\[
{\rm fpr}(x,G/N) = \frac{|C_{\O}(x)|}{|\O|} = \frac{|x^G \cap N|}{|x^G|}
\]
be the fixed point ratio of $x \in G$, where $C_{\O}(x)$ is the set of fixed points of $x$ on $\O = G/N$. As explained in Section \ref{ss:prob}, if we can establish the existence of a constant $c>0$ such that ${\rm fpr}(x_i,G/N) \leqs |x_i^G|^{-c}$ for all $i$, then 
\[
\mathcal{Q}(G,N,2) \leqs \sum_{i=1}^m|x_i^G|\cdot {\rm fpr}(x_i,G/N)^2 \leqs \eta_G(2c-1)
\]
and thus $b(G,N) = 2$ if $\eta_G(2c-1)<1$. With this observation in hand, a key result is \cite[Proposition 2.2]{B07}, which states that $\eta_G(1/3)<1$ if $n \geqs 6$, where $n$ is the dimension of the natural module for $G$. For $n \geqs 6$, we are therefore interested in establishing an upper bound of the form
\[
{\rm fpr}(x,G/N) < |x^G|^{-\frac{2}{3}}
\]
for each $x \in G$ of prime order, which in turn requires a careful analysis of $|x^G \cap N|$ and $|x^G|$. There is an extensive literature on the conjugacy classes of prime order elements in $G$ and the main challenge here is to derive an effective upper bound on $|x^G \cap N|$. The low rank classical groups with $n<6$ will require special attention and they are handled separately.

\vs

\noindent \textbf{Notation.} Let $G$ be a finite group and let $n$ be a positive integer. We will write $C_n$, or just $n$, for a cyclic group of order $n$ and $G^n$ will denote the direct product of $n$ copies of $G$. If $X$ is a subset of $G$, then $i_n(X)$ is the number of elements in $X$ of order $n$. An unspecified extension of $G$ by a group $H$ will be denoted by $G.H$; if the extension splits then we may write $G{:}H$. We adopt the standard notation for simple groups of Lie type from \cite{KL} and all logarithms in this paper are in base $2$.

\vs

\noindent \textbf{Acknowledgements.} Burness thanks the Department of Mathematics at the University of Padua for their generous hospitality during a research visit in autumn 2021. Thomas is supported by EPSRC grant EP/W000466/1. 

\section{Preliminaries}\label{s:prel}

In this section we present some preliminary results which will be needed in the proof of Theorem \ref{t:main}. 

\subsection{Probabilistic methods}\label{ss:prob}

Let $G \leqs {\rm Sym}(\O)$ be a finite transitive permutation group of degree $d$ with point stabiliser $N$ and base size $b(G,N)$. Since the elements of $G$ are uniquely determined by their action on a base, it follows that $|G| \leqs d^{b(G,N)}$ and we obtain the lower bound $b(G,N) \geqs \log_d|G|$. 

Let $c$ be a positive integer and let $\mathcal{Q}(G,N,c)$ be the probability that a random $c$-tuple of elements in $\O$ do not form a base for $G$. Although it is difficult to compute $\mathcal{Q}(G,N,c)$ precisely, a powerful approach for determining an effective upper bound was introduced by Liebeck and Shalev in \cite{LSh}. Indeed, it is straightforward to show that 
\[
\mathcal{Q}(G,N,c) \leqs \sum_{i=1}^m |x_i^G|\cdot {\rm fpr}(x_i,G/N)^c =: \what{\mathcal{Q}}(G,N,c),
\]
where $x_1, \ldots, x_m$ form a complete set of representatives of the conjugacy classes in $G$ of elements of prime order and
\[
{\rm fpr}(x_i,G/N) = \frac{|C_{\O}(x_i)|}{|\O|} = \frac{|x_i^G \cap N|}{|x_i^G|}
\]
is the fixed point ratio of $x_i$ on $\O = G/N$ with $C_{\O}(x_i) = \{ \a \in \O \,:\, \a^{x_i} = \a\}$. The following is an immediate consequence.

\begin{prop}\label{p:base}
If $N \ne 1$ and $\what{\mathcal{Q}}(G,N,2)<1$ then $b(G,N) = 2$. 
\end{prop}

The following result (see \cite[Lemma 2.1]{B07}) is a useful tool for estimating $\what{\mathcal{Q}}(G,N,2)$.

\begin{lem}\label{l:calc}
Suppose $x_{1}, \ldots, x_{s}$ represent distinct $G$-classes such that $\sum_{i}{|x_{i}^{G}\cap N|}\leqs A$ and $|x_{i}^{G}|\geqs B$ for all $i$. Then 
\[
\sum_{i=1}^{s} |x_i^{G}| \cdot {\rm fpr}(x_i,G/N)^2 \leqs A^2/B.
\]
\end{lem}

In our analysis of classical groups, we will work with the zeta-type function 
\begin{equation}\label{e:etaG}
\eta_G(t) = \sum_{i=1}^m|x_i^{G}|^{-t}
\end{equation}
where $t \in \mathbb{R}$. The next result (see \cite[Proposition 2.2]{B07}) will be an essential ingredient.

\begin{prop}\label{p:eta}
Let $G$ be an almost simple classical group with natural module of dimension $n \geqs 6$. Then $\eta_G(1/3)<1$.
\end{prop}

\begin{cor}\label{c:zeta}
Let $G \leqs {\rm Sym}(\O)$ be an almost simple transitive group with nontrivial point stabiliser $N$ and socle $G_0$. Assume $G_0$ is a classical group with natural module of dimension $n \geqs 6$. Then $b(G,N) = 2$ if 
\begin{equation}\label{e:fpr}
{\rm fpr}(x,G/N) \leqs |x^{G}|^{-\frac{2}{3}}
\end{equation}
for all $x \in G$ of prime order.
\end{cor}

\begin{proof}
We have
\[
\what{\mathcal{Q}}(G,N,2) \leqs \sum_{i=1}^m |x_i^{G}|^{-\frac{1}{3}} = \eta_G(1/3) < 1
\]
by Proposition \ref{p:eta} and thus $b(G,N) = 2$ by Proposition \ref{p:base}.
\end{proof}

For $x \in G$, it will be convenient to set 
\begin{equation}\label{e:alpha}
\a(x) = \frac{\log |x^G \cap N|}{\log |x^G|},
\end{equation}
noting that the bound in \eqref{e:fpr} holds if and only if $\a(x) \leqs 1/3$. 

\begin{rem}
For some low dimensional classical groups over small fields, it is often possible to establish a stronger form of Proposition \ref{p:eta} with $\eta_G(t)<1$ for some explicit constant $t<1/3$ (see Lemmas \ref{l:lu}, \ref{l:sp1}, \ref{so_odd1} and \ref{so_even1}). In this situation, it then suffices to show that $\a(x) \leqs (1-t)/2$ for all $x \in G$ of prime order.
\end{rem}

\subsection{Maximal tori}\label{ss:tori}

There is a well developed theory of maximal tori in finite groups of Lie type and here we briefly recall some of the key features and results, following \cite[Section~3.3]{Carter85} and \cite[(2.1)--(2.4)]{S83}. 

As in the introduction, let $\bar{G}$ be a simple algebraic group of adjoint type over the algebraic closure $k$ of $\mathbb{F}_p$ and let $F$ be a Steinberg endomorphism of $\bar{G}$ so that $O^{p'}(\bar{G}^F)$ is a simple group of Lie type over $\mathbb{F}_q$. Let $\bar{T}$ be an $F$-stable maximal torus of $\bar{G}$ and observe that there is a natural action of $F$ on the Weyl group $W = N_{\bar{G}}(\bar{T}) / \bar{T}$. We say $w_1, w_2 \in W$ are \emph{$F$-conjugate} if there exists $x \in W$ such that $w_2 = x^F w_1 x^{-1}$. This defines an equivalence relation on $W$ and the equivalence classes are called \emph{$F$-classes}. The \emph{$F$-centraliser} of $w \in W$ is the subgroup
\begin{equation}\label{e:CW}
C_{W,F}(w) = \{ x \in W \,:\, x^F w x^{-1} = w \}
\end{equation}
of $W$. 

\begin{rem} \label{r:Fclasses}
We refer the reader to \cite[Section 22.1]{MT} for more information on the action of $F$ on $W$ and the character group of $\bar{T}$. For our purposes, let us note that the action of $F$ on $W$ is determined by the corresponding automorphism $\rho$ induced by $F$ on the Coxeter diagram of $\bar{G}$ (the non-oriented Dynkin diagram). If $\rho$ is trivial (in which case $\bar{G}^F$ is an untwisted group), then the $F$-classes of $W$ are just the usual conjugacy classes of $W$ and we see that $C_{W,F}(w) = C_W(w)$ is the centraliser of $w$. Similarly, if $\rho$ corresponds to an involutory graph automorphism of the Coxeter diagram and $\bar{G}$ is not of type $D_r$ with $r \geqs 4$ even, then there is a bijection from the $F$-classes of $W$ to the usual conjugacy classes, and we note that $C_{W,F}(w)$ is isomorphic to $C_W(w)$. In the remaining cases, there is a distinction to be made between the $F$-classes of $W$ and the usual conjugacy classes of $W$, and we refer the reader to \cite[Sections 6.5 and 7]{Gager} for more details. Further comments on the case where $G = {\rm P\O}_{2r}^{-}(q)$ and $r \geqs 4$ is even are presented in  Remark \ref{r:minus}.
\end{rem}

A fundamental result (see \cite[Proposition 3.3.3]{Carter85}, for example) states that there is a bijection from the set of $F$-classes of $W$ to the set of $\bar{G}^F$-classes of $F$-stable maximal tori in $\bar{G}$. Fix an $F$-stable maximal torus $\bar{T}_w$ in the $\bar{G}^F$-class corresponding to the $F$-class of $w$. Since any two maximal tori in $\bar{G}$ are conjugate, it follows that $\bar{T}_w = \bar{T}^g$ for some $g \in \bar{G}$ and the $F$-stability of $\bar{T}_w$ implies that $g^Fg^{-1} \in N_{\bar{G}}(\bar{T})$ (indeed, up to $F$-conjugacy we may assume $w$ is the image of $g^Fg^{-1}$ in $W$). By \cite[Proposition~3.3.6]{Carter85} we have
\[
N_{\bar{G}^F}(\bar{T}_w) / \bar{T}_w^F \cong C_{W,F}(w).
\] 
The order of $\bar{T}_w^F$ can be computed using \cite[Proposition~3.3.5]{Carter85} and the following useful upper bound is proved in \cite[(2.4)]{S83}. 

\begin{lem}\label{l:Tsize}
Let $\bar{S}$ be an $F$-stable $\ell$-dimensional torus in $\bar{G}$. Then   
\[ 
|\bar{S}^F| \leqs (q^a+1)^\ell,
\]
where $a=1/2$ if $\bar{G}^F = {}^2B_2(q)$, ${}^2G_2(q)$ or ${}^2F_4(q)$, otherwise $a=1$. 
\end{lem}

Set $G = O^{p'}(\bar{G}^F)$. Then each $\bar{G}^F$-class of $F$-stable maximal tori in $\bar{G}$ intersects $G$ in a single $G$-class and we set $T = G \cap \bar{T}_w^F$. If we now define $d = |\bar{G}^F:G|$, which is the order of the centre of the simply connected version of $\bar{G}$, then $|\bar{T}_w^F:T| = d$ and it follows that 
\[ 
N= N_G(\bar{T}_w) = T.R,
\] 
where $R = C_{W,F}(w)$ as above. Here it is important to note that $N$ may be a proper subgroup of $N_{G}(\bar{T}_w^F)$. For example, if $G = {\rm L}_n(2)$ and $w = 1$, then $\bar{T}_w^F$ is trivial and $N \cong S_n$ is the subgroup of monomial matrices in $G$. 

We close with the following result, which will be useful in our proof of Theorem \ref{t:main} for exceptional groups of Lie type. In the statement, $\rho_R$ denotes the number of reflections in $R$.  

\begin{lem}\label{l:root}
Let $N = T.R$ as above. If $x \in N$ is a root element in $G$, then $p=2$ and
\[
|x^G \cap N| \leqs \rho_R |T|. 
\]
\end{lem}

\begin{proof}
By \cite[Proposition 1.13]{LLS}, if $y \in N_{\bar{G}}(\bar{T})$ is a root element, then $p=2$ and $y$ centralises a subtorus in $\bar{T}$ of codimension $1$. In particular, each root element in $N$ corresponds to a reflection in $R = N_{\bar{G}^F}(\bar{T}_w) / \bar{T}_w^F$, so $N$ contains at most $\rho_R |T|$ root elements and the result follows.  
\end{proof}

\subsection{Computational methods}\label{ss:comp}

In this final preliminary section, we briefly discuss some of the main computational methods we use in the proof of Theorem \ref{t:main}. In order to perform these computations, we use {\sc Magma} \cite{Magma}, version V2.26-12.

There are essentially two different ways in which we apply computational methods. Firstly, if $G$ is a low rank simple group of Lie type defined over a small field, then it may be possible to verify Theorem \ref{t:main} directly, typically by constructing $G$ and $N$ as permutation groups and using random search to find an element $x \in G$ such that $N \cap N^x = 1$. These are the sort of computations we focus on here, with Lemmas \ref{l:small} and \ref{l:excepcomp} as our main results. Note that we do not need to construct $N$ explicitly in all cases; indeed, to conclude that $b(G,N) = 2$, it suffices to show that there exists an overgroup $L$ of $N$ with $L \cap L^x = 1$ for some $x \in G$.  

For other groups where this direct approach is not feasible, we may still be able to use {\sc Magma} to obtain useful information about $G$ and $N$, which can then be combined with the probabilistic techniques described in Section \ref{ss:prob}. For example, see Lemmas \ref{l:lu}, \ref{l:sp1}, \ref{so_odd1} and \ref{so_even1}, where we use {\sc Magma} to show that $\eta_G(t)<1$ for various classical groups $G$ and explicit constants $t<1/3$, which allows us to work with slightly weaker fixed point ratio estimates in order to establish the bound $\widehat{\mathcal{Q}}(G,N,2)<1$. Similarly, for certain exceptional groups we will use {\sc Magma} to help us estimate $\widehat{\mathcal{Q}}(G,N,2)$. For instance, see the proof of Proposition \ref{p:e6}, where the groups $G = E_6^{\e}(2)$ require special attention.

The following lemma establishes Theorem \ref{t:main} for some specific low dimensional classical groups defined over small fields. 

\begin{lem}\label{l:small}
The conclusion to Theorem \ref{t:main} holds if $G$ is one of the following groups:
\[
\begin{array}{c}
{\rm L}_6^{\e}(2), \, {\rm L}_{6}^{\e}(3), \, {\rm L}_7^{\e}(2), \, {\rm L}_8^{\e}(2) \\
{\rm Sp}_6(2), \, {\rm PSp}_6(3), \, {\rm Sp}_6(4), \, {\rm PSp}_6(5), \, {\rm Sp}_8(2), \, {\rm PSp}_8(3) \\
\O_7(3), \, \O_{8}^{\e}(2),\, {\rm P\O}_8^{\e}(3), \, \O_9(3),\, \O_{10}^{\e}(2)
\end{array}
\]
\end{lem}

\begin{proof}
As noted above, we will use {\sc Magma} \cite{Magma} to verify the result, working with a standard permutation representation of $G$ of degree $n \leqs 22204$ (with equality if $G = {\rm U}_6(3)$). Write $N = T.R$ as above and recall from Section \ref{ss:tori} that we can describe the possibilities for $N$ in terms of the Weyl group of $G$ (for example, see the opening paragraphs in Sections \ref{s:lu}--\ref{s:o_even} below). Also note that the precise cyclic structure of $T$ is given in \cite{BuG}.

We start by handling the following possibilities for $(G,N)$: 
\begin{equation}\label{e:list0}
({\rm U}_6(2),3^4{:}S_6), \, ({\rm Sp}_6(2),3^3{:}(S_2 \wr S_3)), \, (\O_8^{+}(2),3^4{:}(2^3{:}S_4)).
\end{equation}
Note that each of these cases appears in Table \ref{tab:main} and we need to prove that $b(G,N) = 3$. First we construct $N$ by observing that $N = N_G(K)$, where $K$ is an elementary abelian normal subgroup of a Sylow $3$-subgroup of $G$ of order $3^4$, $3^3$ and $3^4$, respectively. In each case, we can find elements $x,y \in G$ by random search such that $N \cap N^x \cap N^y = 1$, which implies that  $b(G,N) \leqs 3$ (in the same way, we can identify an element $x \in G$ such that $|N \cap N^x| = 2,4,4$ in the respective cases, so these cases still satisfy the main statement in Vdovin's conjecture). For $G = {\rm Sp}_6(2)$ and $\O_8^{+}(2)$ we have $\log_d|G|>2$, where $d = |G:N|$, whence $b(G,N) = 3$ as claimed. For $G = {\rm U}_6(2)$, we can use {\sc Magma} to find a complete set $S$ of $(N,N)$ double coset representatives in $G$ and we then check that $|NxN|<|N|^2$ for all $x \in S$. This implies that $N \cap N^x \ne 1$ for all $x \in G$, so $b(G,N) \geqs 3$ and the result follows. 

Now let us turn to the remaining cases, where we need to show that $b(G,N)=2$. First consider the special case $T=1$, which only arises when $R=W$ and $G$ is one of the following:
\[
{\rm L}_6(2), \, {\rm L}_7(2), \, {\rm L}_8(2), \, {\rm Sp}_6(2), \, {\rm Sp}_8(2), \, \O_{8}^{+}(2), \, \O_{10}^{+}(2).
\]
In each of these cases, it is straightforward to construct $N = W$ as a subgroup of $G$ by applying some of the in-built functions in {\sc Magma} for handling groups of Lie type. It is then a routine exercise to check that there exists an element $x \in G$ with $N \cap N^x = 1$ and thus $b(G,N) = 2$.  

Finally, let us assume $T \neq 1$, excluding the cases in \eqref{e:list0} handled above. Here we adopt the following uniform approach, which is straightforward to implement in {\sc Magma}. As above, we work with a standard permutation representation of $G$ and we write $|T|_p = p^a$, where $p$ is defined to be the largest prime divisor of $|T|$. We start by constructing a set of representatives of the conjugacy classes of abelian subgroups of $N_G(P)$ of order $p^a$, where $P$ is a Sylow $p$-subgroup of $G$. For each representative $K$, we construct $L=N_G(K)$ and we discard $K$ if $|L|$ is not divisible by $|N|$. We now have a collection of subgroups of the form $N_G(K)$, at least one of which contains a conjugate of $N$ (this is because $N$ normalises $O_p(T)$, which is an abelian subgroup of order $p^a$). Fix a subgroup $L = N_G(K)$ whose order is divisible by $|N|$. We then use random search to see if there exists an element $x \in G$ with $L \cap L^x = 1$. If we find such an element, then we move to the next compatible subgroup of the form $N_G(K)$ and repeat. On the other hand, if the random search is inconclusive, then we construct a set of representatives of the $L$-classes of subgroups $J$ of $L$ of order $|N|$ and we use random search once again to show that $b(G,J) = 2$. In this way, one can check that $b(G,N) = 2$ in every case. 
\end{proof}

For the exceptional groups of Lie type, we present the following result.

\begin{lem}\label{l:excepcomp}
Let $G$ be one of the following groups
\[
G_2(3), \, G_2(4), \, G_2(5), \, {}^3D_4(2), \, {}^3D_4(4), \, {}^2F_4(2)'
\]
and let $N = N_G(\bar{T})$ as above. Then $b(G,N) = 2$.      
\end{lem}

\begin{proof}
Once again we use {\sc Magma} \cite{Magma}. Let $\mathcal{A}$ be the set of orders of the subgroups $N$ we need to consider. For instance, if $G = G_2(q)$, then $W = D_{12}$ has $6$ conjugacy classes (which coincide with the $F$-classes of $W$; see Remark \ref{r:Fclasses}) and we see that $N$ is one of the following
\[
(C_{q \pm 1})^2.D_{12}, \, C_{q^2\pm q+1}.6, \, C_{q^2-1}.2^2,
\]
noting that $G$ has two conjugacy classes of maximal tori of the form $C_{q^2-1}$.   
So for example, if $q=5$ then $\mathcal{A} = \{96, 126, 186, 192, 432\}$. The corresponding information for the relevant twisted groups can be found in \cite[Table 1.1]{DerMic} for ${}^3D_4(q)$ and \cite{Shin} for ${}^2F_4(2)$.

Suppose $G \ne {}^3D_4(4), {}^2F_4(2)'$ and $|N| = n \in \mathcal{A}$. First we construct $G$ as a permutation group via the function \texttt{AutomorphismGroupOfSimpleGroup} and we then identify a set of  representatives of the conjugacy classes of order $n$ subgroups of $G$. It is then straightforward to find a random element $x \in G$ such that $H \cap H^x = 1$ for each representative $H$ and we conclude that $b(G,N) = 2$. 

A very similar argument applies when $G = {}^2F_4(2)'$. Here it is convenient to work in $L = {}^2F_4(2) = G.2$, noting that the possibilities for $|N_L(\bar{T})|$ can be read off from \cite[Table III]{Shin} (also see \cite[Table 7.3]{Gager}). As before, we construct a set of representatives of the $L$-classes of subgroups $H$ of order $|N_L(\bar{T})|$ and we then use random search to find an element $x \in G$ with $H \cap H^x = 1$. Once again, this allows us to deduce that $b(G,N) = 2$ in all cases.

Finally, let us assume $G = {}^3D_4(4)$. As above, we can construct $G$ as a permutation group of degree $328965$, but the \texttt{Subgroups} function is ineffective and so a slightly different approach is needed. Let $T$ be a maximal torus of $G$. If $T = (C_{21})^2$, $(C_{13})^2$ or $C_{241}$ then $N$ is a maximal subgroup of $G$ and thus $b(G,N) = 2$ by \cite[Proposition 4.2]{BTh_ep}. So to complete the proof, we may assume $T = C_{65} \times C_5$, $C_{63} \times C_3$, $C_{315}$ or $C_{195}$. To handle these cases, we first construct a Sylow $3$-subgroup $H$ and a Sylow $5$-subgroup $K$ of $G$, working with the given permutation representation of $G$. Now $H$ has an element $x$ of order $3$ such that $C_G(x) = {\rm L}_2(64) \times C_3$ contains the maximal tori $C_{63} \times C_3$ and $C_{195}$. Similarly, we find an element $y \in K$ of order $5$ such that $C_G(y) = {\rm L}_2(64) \times C_5$ contains $C_{65} \times C_5$ and $C_{315}$. We can now construct $N$ by taking the normalisers of these tori and it is easy to check that $b(G,N) = 2$ by random search. 
\end{proof}

\section{Linear and unitary groups}\label{s:lu}

In this section, we prove Theorem \ref{t:main} for the classical groups with socle ${\rm L}_{n}^{\e}(q)$. The low dimensional groups with $n<6$ require special attention and they will be treated separately at the end of the section. We begin by recording some preliminary observations.

Write $G = O^{p'}(\bar{G}^F)$, where $\bar{G} = {\rm PSL}_{n}(k)$ is the ambient simple algebraic group defined over the algebraic closure $k$ of $\mathbb{F}_p$ and $F$ is a suitable Steinberg endomorphism of $\bar{G}$. Fix an $F$-stable maximal torus $\bar{T}$ of $\bar{G}$ and set $\bar{N} = N_{\bar{G}}(\bar{T}) = \bar{T}.W$, where $W = S_n$ is the Weyl group of $\bar{G}$. We may assume $\bar{T}$ is the image (modulo scalars) of the group of diagonal matrices in ${\rm SL}_n(k)$ with respect to a standard basis $\{e_1, \ldots, e_n\}$ of the natural module $\bar{V}$ for ${\rm SL}_n(k)$. If we set $\bar{V}_i = \la e_i \ra$, then $\bar{N}$ is the stabiliser in $\bar{G}$ of the direct sum decomposition $\bar{V} =  \bar{V}_1 \oplus \cdots \oplus \bar{V}_n$ and we will abuse notation by writing $(\l_1, \ldots, \l_n)\s$ for a typical element of $\bar{N}$, where $\l_i \in k^{\times}$ and $\s \in S_n$.
 
Recall from Section \ref{ss:tori} that there is a bijection from the set of $\bar{G}^F$-classes of $F$-stable maximal tori in $\bar{G}$ to the set of $F$-classes in $W$. As recorded in Remark \ref{r:Fclasses}, the $F$-classes in $W$ are in bijection with the usual conjugacy classes in this case, which are in turn parameterised by the set of partitions of $n$. Let $w \in W$ be a permutation with cycle-shape $\mu = (n^{a_n}, \ldots, 1^{a_1})$, where $a_{\ell} \geqs 0$ is the multiplicity of $\ell$ as a part of $\mu$. Then in the notation of Section \ref{ss:tori}, the $W$-class of $w$ corresponds to an $F$-stable maximal torus $\bar{T}_w$ of $\bar{G}$ and we set 
$N = N_G(\bar{T}_w) = T.R$. Here $T$ is the image in $G$ of the intersection $\widehat{T} \cap {\rm SL}_{n}^{\e}(q)$, where 
\[
\widehat{T} = \left(C_{q^{n}-\e^{n}}\right)^{a_n} \times \cdots \times \left(C_{q-\e}\right)^{a_1}
\]
is a maximal torus of ${\rm GL}_{n}^{\e}(q)$ (the precise cyclic structure of $T$ is given in \cite[Section 2]{BuG}). In addition, 
\[
R = C_{W,F}(w) \cong C_W(w) = \left(C_n \wr S_{a_n}\right) \times \cdots \times \left(C_1 \wr S_{a_1}\right).
\]

We refer the reader to \cite[Chapter 3]{BG} for detailed information on prime order elements and their conjugacy classes in the finite classical groups. We will also work with the following parameter, noting that bounds on $|x^G|$ in terms of $\nu(x)$ are presented in \cite[Section 3]{fpr2}. 

\begin{defn}\label{d:nu}
Let $G$ be a finite simple classical group over $\mathbb{F}_q$ with natural module $V$ and set $\bar{V} = V \otimes k$, where $k$ is the algebraic closure of $\mathbb{F}_q$. Given $x \in G$, let $\hat{x} \in {\rm GL}(V)$ be a pre-image of $x$ and set
\[
\nu(x) = \min\left\{\dim \, [\bar{V},\l \hat{x}] \,:\, \l \in k^{\times}\right\},
\]
where $[\bar{V},y] = \la v - vy \,:\, v \in \bar{V}\ra$. Note that $\nu(x)$ is equal to the codimension of the largest eigenspace of $\hat{x}$ on $\bar{V}$.
\end{defn}

\begin{rem}\label{r:disc}
Let $G = {\rm L}_n^{\e}(q) = L/Z$, where $L = {\rm SL}_{n}^{\e}(q)$ and $Z = Z(L)$. 
Let $x \in G$ be an element of prime order $r$. By \cite[Lemma 3.11]{fpr2}, if $r$ is odd then either 
\begin{itemize}\addtolength{\itemsep}{0.2\baselineskip}
\item[{\rm (a)}] $x = Z\hat{x}$, where $\hat{x} \in L$ has order $r$; or
\item[{\rm (b)}] $r$ divides $(n,q-\e)$, $C_{\bar{G}}(x)$ is disconnected and 
\begin{equation}\label{e:lu_disc}
|x^G| > \frac{1}{2r}\left(\frac{q}{q+1}\right)^{r-1}q^{n^2\left(1-\frac{1}{r}\right)}.
\end{equation}
\end{itemize}
It is straightforward to see that the same conclusion also holds when $r=2$. In particular, if $r \ne p$ and $C_{\bar{G}}(x)$ is connected, then $x$ lifts to an element of order $r$ in $L$.
\end{rem}

Recall that if $X$ is a subset of $G$ and $r$ is a positive integer, then $i_r(X)$ denotes the number of elements of order $r$ in $X$. The following upper bound will be useful.

\begin{lem}\label{l:lu1}
Let $G = {\rm L}_n^{\e}(q)$ and define $N = T.R$ as above. Then 
\[
i_r(N) \leqs \sum_{j=0}^{\lfloor n/r \rfloor} \frac{n!}{j!(n-jr)!r^j}(q+1)^{n-1-j}
\]
for every prime divisor $r$ of $|N|$.
\end{lem}

\begin{proof}
Recall that $N = T.R$, where $T = G \cap \bar{T}_w^{F}$, $R = C_{W,F}(w) \cong C_W(w)$ and $\bar{T}_w = \bar{T}^g$ for some $g \in \bar{G}$. Set $\Lambda = \{\s \in R \, : \, \s^r = 1\}$ and observe that
\[
i_r(N) \leqs \sum_{\s \in \Lambda} i_r(T\s).
\]
Fix $\s \in \Lambda$ with cycle-shape $(r^j,1^{n-jr})$ as an element of $W = S_n$. We claim that 
\[
i_r(T\s) \leqs (q+1)^{n-1-j}.
\]

To see this, first observe that $i_r(T\s) = |A|$, where $A = \{s \in gTg^{-1} \,:\, |s\s'|=r\}$ and $\s' = g\s g^{-1}$. Now $A$ is contained in $\{s \in \bar{T} \,:\, |s\s'|=r\}$ and by 
considering the action of $\s'$ on $\bar{T}$ we deduce that $A$ is contained in an $F$-stable subtorus $\bar{S}$ of $\bar{T}$ with $\dim \bar{S} = n-1-j$. For example, if 
\[
\s' = (1, \ldots, r)(r+1, \ldots, 2r) \ldots ((j-1)r+1, \ldots, jr)
\]
as an element of $S_n$, then we may identify $\bar{S}$ with the image of the standard maximal torus of ${\rm SL}_r(k)^j \times {\rm SL}_{n-jr}(k) < {\rm SL}_n(k)$, which has dimension $j(r-1)+(n-jr-1) = n-1-j$. Therefore, $|A| \leqs (q+1)^{n-1-j}$ by Lemma \ref{l:Tsize} and this justifies the claim. The result now follows since there are at most 
\[
\frac{n!}{j!(n-jr)!r^j}
\]
elements in $R$ with cycle-shape $(r^j,1^{n-jr})$.
\end{proof}

Recall that if $n \geqs 6$ then Proposition \ref{p:eta} states that $\eta_G(1/3)<1$, where 
$\eta_G(t)$ is the function defined in \eqref{e:etaG}. For certain low dimensional groups over small fields we can use a computational approach to show that $\eta_G(t)<1$ for some smaller constant $t<1/3$. 

\begin{lem}\label{l:lu}
Let $G = {\rm L}_n^{\e}(q)$, where $n \geqs 6$.
\begin{itemize}\addtolength{\itemsep}{0.2\baselineskip}
\item[{\rm (i)}] If $(n,q)$ is one of the following, then $\eta_G(t)<1$ where $t$ is defined as in the table: 
\[
\begin{array}{r|cccccccc}
  & n = 7 & 8 & 9 & 10 & 11 & 12 & 13 & 14 \\ \hline
q=2 & 21/100 & 9/50 & 3/25 & 7/50 & 3/25 & 2/25 & 1/10 & 9/100 \\
3 & 7/50 & 11/100 & & & & &
\end{array}
\]
\item[{\rm (ii)}] If $n = 6$ and $q \leqs 7$, then $\eta_G(t)<1$ where $t$ is defined as follows:
\[
\begin{array}{r|ccccc}
q & 2 & 3 & 4 & 5 & 7 \\ \hline 
t & 1/5 & 4/25 & 17/100 & 13/100 & 3/25  
\end{array}
\]
\end{itemize}
\end{lem}

\begin{proof}
We use {\sc Magma} \cite{Magma}. Set $d = (n,q-\e)$ and let $r$ be a prime divisor of $|G|$. 

First assume $(d,r) = 1$. We take the standard matrix representation of $L = {\rm SL}^{\e}_n(q)$ and we use the functions \texttt{Classes} ($\e=+$) and \texttt{ClassicalClasses} ($\e=-$) to determine the list of conjugacy class sizes of elements of order $r$ in $L$, noting that this coincides precisely with the corresponding list of class sizes in $G$. 

Now suppose $r$ divides $d$ and set $Z = Z(L)$, so $G = L/Z$. As above, we first determine the sizes of the $L$-classes of the form $x^L$, where $x \in L\setminus Z$ and $x^r \in Z$. Fix a class $x^L$ and let $s$ be the multiplicity of $|x^L|$ in this list of class sizes. If $s$ is divisible by $r$, then $G$ has exactly $s/r$ classes of elements of order $r$ with size $|x^L|$, otherwise $G$ has $s$ such classes of size $|x^L|/r$ (see \cite[Propositions 3.2.2, 3.3.3]{BG}, for example). 

In this way, we obtain the complete list $a_1, \ldots, a_m$ of conjugacy class sizes of elements of prime order in $G$, and it is now a routine exercise to verify the bound $\sum_i a_i^{-t}<1$ for the given value of $t$.
\end{proof}

\begin{thm}\label{t:lu}
Suppose $G = {\rm L}_{n}^{\e}(q)$ and $n \geqs 6$. Then $b(G,N) \leqs 3$, with equality if and only if $G = {\rm U}_6(2)$ and $N = 3^4{:}S_6$, in which case $|N \cap N^x| = 2$ for some $x \in G$.
\end{thm}

\begin{proof}
In view of Lemma \ref{l:small}, we may assume
\begin{equation}\label{e:lu}
(n,q) \not\in \{(6,2), (6,3), (7,2), (8,2)\}.
\end{equation}
For the remaining groups, we set $t = 1/3$, with the exception of the specific low dimensional groups over small fields appearing in Lemma \ref{l:lu}, where we define $t<1/3$ as in the lemma. As explained in Section \ref{ss:prob}, if  
\begin{equation}\label{e:alphabd}
\a(x) \leqs \frac{1}{2}(1-t)
\end{equation}
for all $x \in N$ of prime order, where $\a(x)$ is defined as in \eqref{e:alpha}, then  
\[
\mathcal{Q}(G,N,2) \leqs \eta_G(t) < 1
\]
and thus $b(G,N) = 2$. Therefore, our goal is to verify the bound in \eqref{e:alphabd}.

Fix an element $x \in N$ of prime order $r$ and let $\omega \in k$ be a primitive $r$-th root of unity. We now divide the argument into several cases and we freely adopt the notation introduced at the start of Section \ref{s:lu}.

\vs

\noindent \emph{Case 1. $r=p$.}

\vs 

First assume $r=p$, so $x$ is unipotent and $p \leqs n$ (since $p$ must divide $|W|=n!$). Here $x$ is $\bar{G}$-conjugate to an element in $\bar{N}$ of the form $(1, \ldots, 1)\s$, where $\s \in S_n$ has cycle-shape $(p^h,1^{n-hp})$ for some positive integer $h$. Then $x$ has Jordan form $(J_p^h,J_1^{n-hp})$ on $V$ and thus
\[
|x^G| = \frac{|{\rm GL}_{n}^{\e}(q)|}{q^{h(2n-hp-h)}|{\rm GL}_{h}^{\e}(q)||{\rm GL}_{n-hp}^{\e}(q)|} > \frac{1}{2}\left(\frac{q}{q+1}\right)q^{h(p-1)(2n-hp)}.
\]
By arguing as in the proof of Lemma \ref{l:lu1} we deduce that
\begin{equation}\label{e:lunip}
|x^G \cap N| \leqs \frac{n!}{h!(n-hp)!p^h}(q+1)^{h(p-1)}
\end{equation}
and by combining this with the lower bound on $|x^G|$ we get $\a(x) \leqs \b(x)$, where $\b(x)$ is a function of $n$, $h$, $p$ and $f$, where $q=p^f$. So in order to establish the bound in \eqref{e:alphabd}, it suffices to show that $\b(x) \leqs (1-t)/2$. Now if we fix $h$, $p$ and $f$, then $\b(x)$ is a decreasing function of $n$. Therefore, we may assume $n = hp$. Then the corresponding expression for $\b(x)$ is decreasing in both $f$ and $h$, so we may additionally assume that $f=h=1$ and thus $n=p$. Now $\b(x)$ is a decreasing function of $p$ and one checks that $\b(x) \leqs 1/3$ for $p=5$. So for the remainder, we may assume $p \in \{2,3\}$. 

Suppose $p=3$. If $h \geqs 3$ then $\b(x)$ is maximal when $(n,h,f) = (9,3,1)$ and it is easy to check that $\b(x) \leqs 1/3$. Similarly, if $h \in \{1,2\}$ then we reduce to the case $f=1$. Here $n \geqs 7$ (see \eqref{e:lu}) and once again it is straightforward to verify the bound $\b(x) \leqs (1-t)/2$ (note that $\b(x)$ is maximal when $n=7$). 

Finally, suppose $p=2$. If $f \geqs 2$ then one can check that $\b(x) \leqs (1-t)/2$, so we may assume $f=1$ and $n \geqs 9$ (see \eqref{e:lu}). Working with $\b(x)$, and recalling that $t$ is defined in Lemma \ref{l:lu} for $n \leqs 14$, we may assume $(n,h) = (9,1)$, $(9,2)$ or $(10,1)$. In each of these cases, we can establish the desired bound by working with a precise expression for $|x^G|$. For example, if $(n,h) = (9,2)$ then $t = 3/25$ and
\[
|x^G \cap N| \leqs \frac{9!}{2!5!2^2}3^2 = 3402, \;\; |x^G| = \frac{|{\rm GL}_{9}^{\e}(2)|}{2^{24}|{\rm GL}_{2}^{\e}(2)||{\rm GL}_{5}^{\e}(2)|} \geqs 236251890,
\]
which implies that $\a(x) \leqs (1-t)/2$.

\vs

\noindent \emph{Case 2. $r \ne p$, $C_{\bar{G}}(x)$ disconnected.}

\vs 

Here $r$ divides $n$, \eqref{e:lu_disc} holds and one can check that the trivial bound 
\[
|x^G \cap N| \leqs i_r(N) \leqs |N| \leqs (q+1)^{n-1}n!
\]
is sufficient if $n \geqs 22$. For $n \leqs 21$ we can evaluate the upper bound on $i_r(N)$ in Lemma \ref{l:lu1} and this yields $\a(x) \leqs (1-t)/2$ as required.

\vs

\noindent \emph{Case 3. $r \ne p$, $C_{\bar{G}}(x)$ connected, $(r,|T|)=1$.}

\vs 

This is very similar to Case 1, noting that the connectedness of $C_{\bar{G}}(x)$ implies that $x$ lifts to an element of order $r$ in ${\rm SL}_n^{\e}(q)$ (see Remark \ref{r:disc}). Here $r \leqs n$ and $x$ is $\bar{G}$-conjugate to an element in $\bar{N}$ of the form $(1, \ldots, 1)\s$, where $\s \in S_n$ has cycle-shape $(r^h,1^{n-hr})$ for some $1 \leqs h <n/r$ (note that $h \ne n/r$ since we are assuming $C_{\bar{G}}(x)$ is connected). Therefore, $x$ is $\bar{G}$-conjugate to 
$(I_{n-h(r-1)}, \omega I_{h}, \ldots, \omega^{r-1}I_{h})$, modulo scalars, so 
\begin{equation}\label{e:lu11}
|x^G|>\frac{1}{2}\left(\frac{q}{q+1}\right)^{r-1}q^{h(r-1)(2n-hr)}
\end{equation}
and by arguing as in Case 1 we see that 
\[
|x^G \cap N| \leqs \frac{n!}{h!(n-hr)!r^h}(q+1)^{h(r-1)}.
\]
One can now check that these bounds yield $\a(x) \leqs (1-t)/2$. 

\vs

\noindent \emph{Case 4. $r \ne p$, $C_{\bar{G}}(x)$ connected, $r$ divides $|T|$.}

\vs 

To complete the proof, we may assume $r \ne p$ and $r$ divides $|T|$. Let $0 \leqs h < n/r$ be maximal such that $x$ is $\bar{G}$-conjugate to an element in a coset $\bar{T}\s$, where $\s \in S_n$ has cycle-shape $(r^{h},1^{n-hr})$. 

First assume $h=0$, in which case $x^G \cap N \subseteq T$ and thus  
$|x^G \cap N| \leqs (q+1)^{n-1}$. 
Set $s = \nu(x)$ (see Definition \ref{d:nu}). If $s \geqs 3$ then \cite[Corollary 3.38]{fpr2} gives
\[
|x^G| > \frac{1}{2}\left(\frac{q}{q+1}\right)q^{6n-18}
\]
and we deduce that $\a(x) \leqs (1-t)/2$ as required. Similarly, if $s \in \{1,2\}$ then the bounds
\[
|x^G \cap N| \leqs s\binom{n}{s},\;\; |x^G| > \frac{1}{2}\left(\frac{q}{q+1}\right)q^{2s(n-s)}
\]
are sufficient.

For the remainder, let us assume $h \geqs 1$. Then $r \leqs n$ and each $r$-th root of unity has multiplicity at least $h$ as an eigenvalue of $x$ on $\bar{V}$ (and by the maximality of $h$, the multiplicity of at least one eigenvalue is exactly $h$). One can check that $|x^G|$ is minimal when $x$ is $\bar{G}$-conjugate to $(I_{n-h(r-1)},\omega I_{h}, \ldots, \omega^{r-1}I_{h})$ and thus \eqref{e:lu11} holds. Next observe that there exists an integer $j$ in the range $0 \leqs j \leqs h$ such that $x$ is $\bar{G}$-conjugate to an element in $\bar{N}$ of the form $(z_1, \ldots, z_n)\sigma$, where 
\[
\sigma = (1, \ldots, r) (r+1, \ldots, 2r) \cdots ((j-1)r+1, \ldots, jr) \in S_n
\]
has cycle-shape $(r^j,1^{n-jr})$ and $z_i$ is nontrivial (of order $r$) only if $i>jr$. 
Since there are at most $r$ possibilities for each $z_i$ with $i>jr$, it follows that   
\begin{equation}\label{e:acc}
|x^G \cap N| \leqs \sum_{j=0}^{h}\frac{n!}{j!(n-jr)!r^j}(q+1)^{j(r-1)}r^{n-jr},
\end{equation}
which in turn implies that 
\begin{equation}\label{e:nacc}
|x^G \cap N| \leqs (q+1)^{h(r-1)}r^n \sum_{j=0}^{h} \left(\frac{n^r}{r}\right)^j \leqs 2\left(\frac{n^r}{r}\right)^h(q+1)^{h(r-1)}r^n.
\end{equation}

For now, let us assume $r \geqs 3$ and $q \geqs 3$. By combining the lower bound on $|x^G|$ in \eqref{e:lu11} with the upper bound on $|x^G \cap N|$ in \eqref{e:nacc}, we deduce that $\a(x) \leqs 1/3$ if $n \geqs 22$ or $q \geqs 31$, which means that we may assume $n \leqs 21$ and $q \leqs 29$. Then by evaluating the upper bound in \eqref{e:acc}, we may further assume that $n \leqs 12$ and $q \leqs 5$, with $r \in \{3,5\}$ and $h = 1$.

Suppose $(r,h) = (3,1)$. If $x$ is $\bar{G}$-conjugate to $(I_{n-2},\omega,\omega^2)$ then the bounds
\[
|x^G \cap N| \leqs 2\binom{n}{2}+\frac{n!}{(n-3)!3}(q+1)^2,\;\; |x^G|>\frac{1}{2}\left(\frac{q}{q+1}\right)^2q^{4n-6}
\]
are sufficient. Otherwise, 
\[
|x^G| >\frac{1}{2}\left(\frac{q}{q+1}\right)^2q^{6n-14}
\]
(minimal if $x$ is of the form $(I_{n-3}, \omega I_2, \omega^2)$) and one can check that the upper bound on $|x^G \cap N|$ in \eqref{e:acc} gives $\a(x) \leqs (1-t)/2$. A very similar argument applies if $(r,h) = (5,1)$. Indeed, if $x$ is conjugate to $(I_{n-4},\omega, \omega^2, \omega^3, \omega^4)$ then the bounds
\[
|x^G \cap N| \leqs 4!\binom{n}{4}+\frac{n!}{(n-5)!5}(q+1)^4,\;\; |x^G| > \frac{1}{2}\left(\frac{q}{q+1}\right)^4q^{8n-20}
\]
are good enough. And if $x$ is not of this form, then
\[
|x^G|>\frac{1}{2}\left(\frac{q}{q+1}\right)^4q^{10n-32}
\]
and we obtain $\a(x) \leqs (1-t)/2$ via the upper bound on $|x^G \cap N|$ in \eqref{e:acc}.

To complete the proof, we may assume $r=2$ or $q=2$. We will first deal with the case $r=2$, so $q \geqs 3$ is odd and $x = (-I_h,I_{n-h})$ (modulo scalars) with $\nu(x) = h<n/2$ (recall that we are assuming $C_{\bar{G}}(x)$ is connected). Here  
\[
|x^G|>\frac{1}{2}\left(\frac{q}{q+1}\right)q^{2h(n-h)}
\] 
and by arguing as in the proof of Lemma \ref{l:lu1} we see that
\begin{equation}\label{e:lu2}
|x^G \cap N| \leqs \sum_{j=0}^h\frac{n!}{j!(n-2j)!2^j}(q+1)^j\binom{n-2j}{h-j}.
\end{equation}
In particular, if $h=1$ then 
\[
|x^G \cap N| \leqs n+\frac{n!}{(n-2)!2}(q+1),\;\; |x^G|>\frac{1}{2}\left(\frac{q}{q+1}\right)q^{2n-2}
\]
and we deduce that $\a(x) \leqs (1-t)/2$. Now assume $h \geqs 2$. If $n=6$ then $h=2$ and the above bounds are sufficient. For $n \geqs 7$, one can check that the upper bound on $|x^G \cap N|$ given in \eqref{e:nacc} yields $\a(x) \leqs 1/3$ if $n \geqs 24$ or $q \geqs 23$, so we may assume $n \leqs 23$ and $q \leqs 19$. At this point, we can now switch to the upper bound on $|x^G \cap N|$ in \eqref{e:lu2} and we obtain the desired bound $\a(x) \leqs (1-t)/2$.

Finally, let us assume $q=2$, so $r \geqs 3$ and $n \geqs 9$. There are several cases that require special attention. 

First assume $(r,h) = (3,1)$ and set $s = \nu(x)$. Since $h=1$, it follows that either $\omega$ or $\omega^2$ has multiplicity $1$ as an eigenvalue of $x$ on $\bar{V}$ and therefore
\[
|x^G \cap N| \leqs n \cdot 2^{n-1} + \frac{n!}{(n-3)!3}3^2 \cdot 2^{n-3}.
\]
If $s \geqs 4$ then $|x^G|>\frac{2}{9}2^{8n-26}$ and we deduce that $\a(x) \leqs (1-t)/2$. Similarly, if $s=3$ then we may assume $x = (I_{n-3},\omega I_2, \omega^2)$, so $|x^G|>\frac{2}{9}2^{6n-14}$ and the bound
\[
|x^G\cap N| \leqs n\binom{n-1}{2}+\frac{n!}{(n-3)!3}3^2(n-3)
\]
is sufficient. Finally, if $s=2$ then $x = (I_{n-2},\omega, \omega^2)$, $|x^G|>\frac{2}{9}2^{4n-6}$ and the result follows since
\[
|x^G \cap N| \leqs 2\binom{n}{2}+\frac{n!}{(n-3)!3}3^2.
\]

Next suppose $(r,h) = (5,1)$. Here we observe that $r$ is a primitive prime divisor of $q^4-1$, so the condition $h=1$ implies that $x$ is of the form $(I_{n-4},\omega, \omega^2, \omega^3,\omega^4)$. Therefore,
\[
|x^G \cap N| \leqs 4!\binom{n}{4}+\frac{n!}{(n-5)!5}3^4,\;\; |x^G|>\frac{1}{3}2^{8n-20}
\]
and we deduce that $\a(x) \leqs (1-t)/2$.

For the remainder, we may assume $(r,h) \ne (3,1), (5,1)$. Now
\begin{equation}\label{e:lu3}
|x^G| > \frac{1}{2}\left(\frac{2}{3}\right)^{r-1}2^{h(r-1)(2n-hr)}
\end{equation}
and one can check that the upper bound in \eqref{e:nacc} gives $\a(x) \leqs 1/3$ if $n \geqs 54$. Similarly, if $n \leqs 53$ then the bound in \eqref{e:acc} is sufficient if $n \geqs 31$. Therefore, to complete the proof we may assume $n \leqs 30$.

Suppose $(r,h) = (7,1)$ and note that $r$ is a primitive prime divisor of $q^3-1$. If $x$ is of the form $(I_{n-6}, \omega, \ldots, \omega^6)$ then the bounds
\[
|x^G\cap N| \leqs 6!\binom{n}{6}+\frac{n!}{(n-7)!7}3^6,\;\; |x^G|>\frac{1}{3}2^{12n-42}
\]
are sufficient. On the other hand, if $x$ is not of this form, then $n \geqs 10$, $|x^G|>\frac{1}{3}2^{18n-96}$ and the bound in \eqref{e:acc} yields $\a(x) \leqs (1-t)/2$. The case $(r,h) = (11,1)$ is handled in an entirely similar fashion. 

Next assume $r=3$ and $h \in \{2,3\}$, in which case there are two possibilities for $x$ (up to scalars). If $x = (I_{n-2h},\omega I_h, \omega^2I_h)$ then 
\[
|x^G \cap N| \leqs \sum_{j=0}^h \frac{n!}{j!(n-3j)!3^j}\frac{(n-3j)!}{(n-2h-j)!(h-j)!^2}3^{2j}
\]
and we deduce that $\a(x) \leqs (1-t)/2$ since $|x^G|>\frac{2}{9}2^{2h(2n-3h)}$. Otherwise, $x$ is of the form $(I_{n-2h-1},\omega I_{h+1},\omega^2 I_{h})$ and the bounds
\[
|x^G \cap N| \leqs \sum_{j=0}^h \frac{n!}{j!(n-3j)!3^j}\frac{(n-3j)!}{(n-2h-j-1)!(h+1-j)!(h-j)!}3^{2j}
\]
and 
\[
|x^G| > \frac{2}{9}2^{4nh+2n-6h^2-6h-2}
\]
are sufficient.

Finally, suppose $q=2$, $n \leqs 30$ and $hr \not\in \{3,5,6,7,9,11\}$. In these cases  one can check that the bounds in \eqref{e:acc} and \eqref{e:lu3} are good enough. 
\end{proof}

To complete the proof of Theorem \ref{t:main} for linear and unitary groups, it remains to consider the following low-dimensional groups
\[
{\rm L}_2(q), \, {\rm L}_3^{\e}(q), \, {\rm L}_4^{\e}(q), \, {\rm L}_5^{\e}(q).
\]

\begin{prop}\label{l:psl2}
If $G = {\rm L}_{2}(q)$ then $b(G,N) \leqs 3$, with equality if and only if $q \geqs 4$ is even and $N = D_{2(q+1)}$, in which case $|N \cap N^x| = 2$ for all $x \in G \setminus N$.
\end{prop}

\begin{proof}
Set $d = (2,q-1)$ and note that $N = D_{2(q-\e)/d}$ with $\e = \pm 1$. The groups with $q<13$ can be checked directly using {\sc Magma} \cite{Magma} and so we may assume $q \geqs 13$. Here $N$ is a soluble maximal subgroup of $G$ and \cite[Lemmas 4.7, 4.8]{B20} imply that $b(G,N) \leqs 3$, with equality if and only if $q$ is even and $N = D_{2(q+1)}$. In the latter case, every nontrivial subdegree for the action of $G$ on $G/N$ is $q+1$ (see \cite[Table 2]{FI}, for example), so $|N \cap N^x| = 2$ for all $x \in G \setminus N$ and the result follows.
\end{proof}

\begin{prop}\label{l:psl3}
If $G = {\rm L}_{3}^{\e}(q)$ then $b(G,N) \leqs 3$, with equality if and only if one of the following holds: 
\begin{itemize}\addtolength{\itemsep}{0.2\baselineskip}
\item[{\rm (i)}] $(G,N) = ({\rm L}_3(2), 7{:}3)$, in which case $|N \cap N^x| = 3$ for all $x \in G \setminus N$. 
\item[{\rm (ii)}] $(G,N) = ({\rm U}_3(3), 4^2{:}S_3)$, in which case $|N \cap N^x| = 3$ for some $x \in G$. 
\end{itemize}
\end{prop}

\begin{proof}
Here $W = S_3$ has $3$ conjugacy classes and we observe that the maximal tori in ${\rm SL}_3^{\e}(q)$ are as follows, up to conjugacy:
\begin{equation}\label{e:list1}
C_{q^2+\e q+1},\; C_{q^2-1},\; (C_{q-\e})^2
\end{equation}
and the corresponding possibilities for $N = T.R$ are $C_{q^2+\e q +1}{:}3$, $C_{q^2-1}{:}2$ and $(C_{q-\e})^2{:}S_3$, respectively (modulo scalars). The groups with $q \leqs 5$ can be handled using {\sc Magma} (see the proof of Lemma \ref{l:small}, for example), so we will assume $q \geqs 7$.

In the first and third cases in \eqref{e:list1} we observe that $N$ is a soluble maximal subgroup of $G$ and thus $b(G,N) = 2$ by \cite[Lemmas 6.4, 6.5]{B20}. For the remainder we may assume that $N = \what{N}/Z$, where $\what{N} = C_{q^2-1}{:}2$ and $Z = Z({\rm SL}_3^{\e}(q))$. We may identify $\what{N}$ with the normaliser in ${\rm GL}_{2}^{\e}(q)$ of an element of order $q^2-1$. Set $d = |Z| = (3,q-\e)$. By Proposition \ref{p:base}, it suffices to show that $\what{\mathcal{Q}}(G,N,2)<1$.

Let $x \in N$ be an element of prime order $r$. First assume $r=2$ and note that $G$ has a unique conjugacy class of involutions, so $|x^G \cap N| = i_2(\what{N})$. Write $\mathbb{F}_{q^2}^{\times} = \la \l \ra$ and identify $\what{N}$ with $\la \l \ra {:} \la \phi \ra$, where $\phi: \l \mapsto \l^{\e q}$. It is straightforward to show that the coset $\la \l \ra\phi$ contains precisely $q+\e$ involutions, whence $|x^G \cap N| \leqs q+2 = a_1$ and we note that $|x^G| \geqs (q^3+1)(q-1) = b_1$ (minimal if $p=2$ and $\e = -$). 

Now assume $r$ is odd, so $r$ divides $q^2-1$. If $x$ is regular (that is, if $C_{\bar{G}}(x)^0$ is a maximal torus), then 
\[
|x^G| \geqs \frac{|{\rm GU}_3(q)|}{3|{\rm GU}_{1}(q)|^3} = \frac{1}{3}q^3(q-1)(q^2-q+1) = b_2
\]
and we note that $|N| \leqs 2(q^2-1) = a_2$. On the other hand, if $x$ is non-regular then $r$ divides $q-\e$, $|x^G \cap N| = 1$ and
\[
|x^G|  = \frac{|{\rm GL}_3^{\e}(q)|}{|{\rm GL}_{2}^{\e}(q)||{\rm GL}_{1}^{\e}(q)|} \geqs q^2(q^2-q+1) = b_3.
\]
In view of Lemma \ref{l:calc}, it follows that the combined contribution to $\what{\mathcal{Q}}(G,N,2)$ from non-regular semisimple elements of odd order is at most 
\[
\sum_{r \in \pi} (r-1)/b_3 < q\log (q+1)/b_3 = c,
\]
where $\pi$ is the set of odd prime divisors of $q-\e$ (here we are using the fact that $r \leqs q+1$ and $|\pi| \leqs \log(q+1)$).

By bringing together the above estimates, using Lemma \ref{l:calc} once again, we obtain
\[
\what{\mathcal{Q}}(G,N,2) < a_1^2/b_1 + a_2^2/b_2 + c < 1
\]
for all $q \geqs 7$, so $b(G,N) = 2$ and the result follows.
\end{proof}

\begin{prop}\label{l:psl4}
If $G = {\rm L}_{4}^{\e}(q)$ then either $b(G,N)=2$, or $G = {\rm U}_4(2)$ and $N = 3^3{:}S_4$, with $b(G,N) = 4$ and $|N \cap N^x| \in \{24,54\}$ for all $x \in G \setminus N$. 
\end{prop}

\begin{proof}
There are $5$ conjugacy classes in the Weyl group $W = S_4$ and so there are $5$ classes of maximal tori in ${\rm SL}_{4}^{\e}(q)$. The groups with $q \leqs 5$ can be handled using {\sc Magma}, so we will assume $q \geqs 7$.

If $T$ is the image of the split torus $(C_{q-\e})^3$, then $N$ is a soluble maximal subgroup of $G$ and thus \cite[Lemma 6.6]{B20} gives $b(G,N) = 2$. For the remainder, we may assume $|N| \leqs 8(q+1)^3/d$, where $d = (4,q-\e)$. As in the proof of the previous proposition, our aim is to verify the bound $\what{\mathcal{Q}}(G,N,2)<1$. Let $x \in N$ be an element of prime order $r$.

First assume $r \geqs 5$ and observe that $x$ is semisimple and $x^G \cap N \subseteq T$. If $\nu(x) \geqs 2$ then 
\[
|x^G| \geqs \frac{|{\rm GU}_4(q)|}{|{\rm GU}_2(q)|^2} = q^4(q^2-q+1)(q^2+1) = b_1
\]
and we note that there are at most $a_1 = (q+1)^3$ such elements in $N$ (since $|T| \leqs (q+1)^3$). On the other hand, if $\nu(x) = 1$ then $r$ must divide $q-\e$ and we observe that $|x^G \cap N| \leqs 4$ and
\[
|x^G| = \frac{|{\rm GL}_{4}^{\e}(q)|}{|{\rm GL}_{3}^{\e}(q)||{\rm GL}_{1}^{\e}(q)|} \geqs q^3(q^2+1)(q-1) = b.
\]
Therefore, the total contribution to $\what{\mathcal{Q}}(G,N,2)$ from these elements is at most 
\[
\sum_{r \in \pi}4^2(r-1)/b < 16q\log(q+1)/b = c,
\]
where $\pi$ is the set of odd prime divisors $r \geqs 5$ of $q-\e$. 

Next assume $r = 3$. If $x^G \cap N \not\subseteq T$ then $x$ is of the form $(I_2, \omega, \omega^2)$ if $p \ne 3$ and $(J_3,J_1)$ if $p=3$, so 
\[
|x^G| \geqs \frac{|{\rm GU}_{4}(q)|}{|{\rm GU}_{2}(q)||{\rm GU}_{1}(q)|^2} = q^5(q^2-q+1)(q^2+1)(q-1) = b_2
\]
and we have the trivial bound $|x^G \cap N| \leqs |N| \leqs 8(q+1)^3/d = a_2$. On the other hand, if $x^G \cap N \subseteq T$ then $|x^G| \geqs q^3(q^2+1)(q-1) = b_3$ and we note that $i_3(T) \leqs 3^3-1 = a_3$ since $T$ is the direct product of at most  three cyclic groups.

Finally, let us assume $r=2$. If $\nu(x) = 1$ then $|x^G| \geqs q^3(q^2+1)(q-1) = b_4$ and we calculate that $N$ contains at most $(2^3-1) + 2\binom{4}{2}(q+1) = 12q+19 = a_4$ involutions of this form. Similarly, if $\nu(x) = 2$ then $|x^G| \geqs \frac{1}{2}q^4(q-1)(q^3-1) = b_5$ and by evaluating the upper bound in Lemma \ref{l:lu1} we obtain $|x^G \cap N| \leqs (q+1)(q^2+8q+10) = a_5$.

In conclusion, we have 
\[
\what{\mathcal{Q}}(G,N,2) < \sum_{i=1}^5a_i^2/b_i + c < 1
\]
for all $q \geqs 7$ and the result follows.
\end{proof}

\begin{prop}\label{l:psl5}
If $G = {\rm L}_{5}^{\e}(q)$ then $b(G,N) \leqs 3$, with equality if and only if $G = {\rm U}_5(2)$ and $N = 3^4{:}S_5$, in which case $|N \cap N^x| \in \{12, 36, 72, 120, 324\}$ for all $x \in G \setminus N$. 
\end{prop}

\begin{proof}
Here $W = S_5$ and so there are $7$ conjugacy classes of maximal tori to consider. The groups with $q \leqs 5$ can be handled using {\sc Magma}, so we will assume $q \geqs 7$. As before, our aim is to verify the bound $\what{\mathcal{Q}}(G,N,2)<1$. Let $x \in N$ be an element of prime order $r$. 

First observe that the contribution to $\what{\mathcal{Q}}(G,N,2)$ from the elements $x \in G$ with $|x^G|>\frac{1}{2}q^{14} = b_1$ is less than $a_1^2/b_1$, where $a_1 = 120(q+1)^4$ is an upper bound on the order of $N$. For the remainder, we may assume $|x^G| \leqs \frac{1}{2}q^{14}$. If $r=p$ then $r$ must divide $|W|$, so $r \in \{2,3,5\}$, and by considering the given condition on $|x^G|$ we deduce that $r=2$ and $x$ has Jordan form $(J_2^j,J_1^{5-2j})$ on the natural module, where $j=1$ or $2$. If $j=1$ then $|x^G|>\frac{1}{2}q^8 = b_2$ and \eqref{e:lunip} gives $|x^G \cap N| \leqs 10(q+1) = a_2$. Similarly, if $j=2$ then $|x^G|>\frac{1}{2}q^{12} = b_3$ and $|x^G \cap N| \leqs 15(q+1)^2= a_3$. 

Now assume $r \ne p$ (we continue to assume that $|x^G| \leqs \frac{1}{2}q^{14}$). Suppose $x^G \cap N \subseteq T$. If $\nu(x) \geqs 2$ then $|x^G|>\frac{1}{2}q^{12} = b_4$ and we note that $|T| \leqs (q+1)^4 = a_4$. On the other hand, if $\nu(x) = 1$ then $r$ must divide $q-\e$ and we have $|x^G \cap N| \leqs 5$ and $|x^G|>\frac{1}{2}q^8 = b$. Therefore, the combined contribution from semisimple elements $x$ with $x^G \cap N \subseteq T$ and $\nu(x) = 1$ is less than
\[
\sum_{r \in \pi}(r-1) \cdot 5^2/b < 25q\log(q+1)/b = c,
\]
where $\pi$ is the set of prime divisors of $q-\e$. Finally, let us assume $x^G \cap N \not\subseteq T$, so $r \in \{2,3,5\}$. In fact, the assumption $|x^G| \leqs \frac{1}{2}q^{14}$ implies that $r=2$ and $x = (-I_h,I_{5-h})$ with $h \in \{1,2\}$. If $h=1$ then $|x^G|>\frac{1}{2}q^8 = b_5$ and we compute $|x^G \cap N| \leqs 5+\binom{5}{2}(q+1) = 10q+15 = a_5$. And if $h=2$ we get $|x^G|>\frac{1}{2}q^{12} = b_6$ and \eqref{e:lu2} yields
\[
|x^G \cap N| \leqs \binom{5}{2}+3\binom{5}{2}(q+1) + \frac{5!}{2!1!2^2}(q+1)^2 = 15q^2+60q+55 = a_6.
\]

Putting these estimates together, we deduce that 
\[
\what{\mathcal{Q}}(G,N,2) < \sum_{i=1}^6a_i^2/b_i + c < 1
\]
for $q \geqs 7$ and the result follows.
\end{proof}

\section{Symplectic groups}\label{s:symp}

In this section we prove Theorem \ref{t:main} for symplectic groups. Let $G = O^{p'}(\bar{G}^F) = {\rm PSp}_n(q)'$, where $n=2m \geqs 4$ and $q = p^f$ for some prime $p$. Let $W$ be the Weyl group of $\bar{G}$ and note that $W = S_2 \wr S_m < S_n$ is the hyperoctahedral group. Fix an $F$-stable maximal torus $\bar{T}$ of $\bar{G}$ and set $\bar{N} = \bar{T}.W$. We may assume $\bar{T}$ is the group of diagonal matrices in $\bar{G} = {\rm PSp}_n(k)$ (modulo scalars) with respect to a standard symplectic basis $\{e_1,f_1, \ldots, e_m,f_m\}$ of the natural module $\bar{V}$. Set $\bar{V}_i = \la e_i,f_i\ra$ and observe that $\bar{N}$ is contained in the stabiliser of the orthogonal decomposition $\bar{V} = \bar{V}_1 \perp \cdots \perp \bar{V}_m$. More precisely, $\bar{N}$ is the image (modulo scalars) of $\bar{L} \wr S_m$, where 
\begin{equation}\label{e:barL}
\bar{L} = \left\langle {\rm diag}(\l, \l^{-1}), z \,:\, \l \in k^{\times} \right\rangle < {\rm Sp}_2(k)
\end{equation}
and $z=\left(\begin{smallmatrix} \phantom{-}0 & 1 \\ -1 & 0 \end{smallmatrix}\right)$.

Next recall that the conjugacy classes in $W$ (and hence the $F$-classes as well; see Remark \ref{r:Fclasses}) are parameterised by pairs of partitions $(\l,\mu)$ with $|\l|+|\mu|=m$ (here $|\l| \geqs 0$ denotes the sum of the parts comprising $\l$, and similarly for $|\mu|$). Fix an element $w \in W$ corresponding to the pair $(\l,\mu)$, where $\l = (m^{a_m}, \ldots, 1^{a_1})$ and $\mu = (m^{b_m}, \ldots, 1^{b_1})$. Then in the notation of Section \ref{ss:tori}, the $W$-class of $w$ corresponds to an $F$-stable maximal torus $\bar{T}_w$ of $\bar{G}$ and we set $N = N_G(\bar{T}_w) = T.R$. Here $R \cong C_W(w)$ and $T$ is the image (modulo scalars) of a maximal torus 
\[
\what{T} = \left(C_{q^{m}-1}\right)^{a_m} \times \cdots \times \left(C_{q-1}\right)^{a_1} \times \left(C_{q^{m}+1}\right)^{b_m} \times \cdots \times \left(C_{q+1}\right)^{b_1}
\]
of ${\rm Sp}_n(q)$. See \cite[Theorem 3]{BuG} for the precise cyclic structure of $T$.

\begin{lem}\label{l:sp0}
Let $G = {\rm PSp}_{n}(q)$, where $n=2m$ and $N = T.R$ is defined as above. Then 
\[
i_r(N) \leqs \sum_{j=0}^{\lfloor m/r \rfloor} 2^{j(r-1)}\frac{m!}{j!(m-jr)!r^j}(q+1)^{m-j}
\]
for every odd prime divisor $r$ of $|N|$.
\end{lem}

\begin{proof}
Suppose $\s \in R$ has order $r$. Then $\s$ is $W$-conjugate to $(1, \ldots, 1)\rho$, where $\rho \in S_m$ has cycle-shape $(r^j,1^{m-jr})$ for some $j \geqs 1$, and we note that
\[
|\s^W| = 2^{j(r-1)}\frac{m!}{j!(m-jr)!r^j}.
\] 
We can now proceed as in the proof of Lemma \ref{l:lu1} and we omit the details.
\end{proof}

\begin{lem}\label{l:sp1}
Let $G = {\rm PSp}_n(q)$, where $n \geqs 6$. 
\begin{itemize}\addtolength{\itemsep}{0.2\baselineskip}
\item[{\rm (i)}] If $n \leqs 12$ and $q \leqs 7$ then $\eta_G(t)<1$, where $t$ is defined as follows: 
\[
\begin{array}{r|ccccc}
 & q= 2 & 3 & 4 & 5 & 7 \\ \hline
n = 6 & 31/100 & 6/25 & 19/100 & 9/50 & 3/20 \\
8 & 23/100 & 9/50 & 7/50 & 7/50 & 3/25 \\
10 & 9/50 & 7/50 & 11/100 & 11/100 & 9/100 \\
12 & 7/50 & 11/100 & 9/100 & 9/100 & 7/100
\end{array}
\]
\item[{\rm (ii)}] If $n=6$ and $8 \leqs q \leqs 32$ then $\eta_G(t)<1$ for $t = 7/50$.
\item[{\rm (iii)}] If $n \in \{8,10\}$ and $8 \leqs q \leqs 16$ then $\eta_G(t)<1$ for $t = 3/25$.
\item[{\rm (iv)}] If $14 \leqs n \leqs 24$ and $q=2$ then $\eta_G(t)<1$ for $t = 3/25$. 
\end{itemize}
\end{lem}

\begin{proof}
This is very similar to the proof of Lemma \ref{l:lu}, using {\sc Magma} and the function \texttt{ClassicalClasses} to compute the size of every conjugacy class in the matrix group $L = {\rm Sp}_n(q)$ comprising elements of prime order. Note that if $q$ is odd, then to obtain the sizes of the involution classes in $G$ we first compute the sizes of the $L$-classes of the form $x^L$, where $x \in L \setminus Z$ and $x^2 \in Z$ for $Z = Z(L)$. Fix a class $x^L$ and let $s$ be the multiplicity of $|x^L|$ in this list of class sizes. If $s$ is even, then $G$ has $s/2$ classes of involutions with size $|x^L|$, otherwise $G$ has $s$ classes of size $|x^L|/2$ (see \cite[Table B.5]{BG}, for example).
\end{proof}

\begin{thm}\label{t:sp}
Suppose $G = {\rm PSp}_{n}(q)$ and $n \geqs 6$. Then $b(G,N) \leqs 3$, with equality if and only if $G = {\rm Sp}_6(2)$ and $N = 3^3{:}(S_2 \wr S_3)$, in which case $|N \cap N^x| = 4$ for some $x \in G$.
\end{thm}

\begin{proof}
By considering Lemma \ref{l:small}, we may assume
\begin{equation}\label{e:sp}
(n,q) \not\in \{(6,2), (6,3), (6,4), (6,5), (8,2), (8,3)\}
\end{equation}
and so we will exclude these cases for the remainder of the proof. Set $t=1/3$, with the exception of the groups in Lemma \ref{l:sp1}, where we define $t$ as in the lemma. As in the proof of Theorem \ref{t:lu}, it suffices to show that $\a(x) \leqs (1-t)/2$ for all $x \in N$ of prime order and we partition the proof into several cases. Let $\omega \in k$ be a primitive $r$-th root of unity and let us adopt the notation introduced at the beginning of Section \ref{s:symp}.

\vs

\noindent \emph{Case 1. $r=p=2$.}

\vs

Here $x$ is a unipotent involution and we adopt the standard notation from \cite{AS}. In particular, if $x$ has Jordan form $(J_2^h,J_1^{n-2h})$ on $V$, then either $h$ is odd and $x$ is of type $b_h$, or $h$ is even and $x$ is of type $a_h$ or $c_h$. Recall that $z = \left(\begin{smallmatrix} 0 & 1 \\ 1 & 0 \end{smallmatrix}\right) \in \bar{L}$ (see \eqref{e:barL}).

We begin by describing the involutions in the algebraic group $\bar{N} = N_{\bar{G}}(\bar{T}) = \bar{L} \wr S_m$. Clearly, there are no involutions in $\bar{T}$, while every involution in $\bar{L}$ is conjugate to $z$, which is of type $b_1$ as an element of ${\rm Sp}_2(k)$. Therefore, if $x = (z_1, \ldots, z_m) \in \bar{L}^m$ is an involution with $\ell$ nontrivial components, then $x$ is of type $b_{\ell}$ if $\ell$ is odd, otherwise $x$ is of type $c_{\ell}$ (see \cite[Lemma 3.4.14]{BG}). If $x \in \bar{N} \setminus \bar{L}^m$ is an involution, then there exists $1 \leqs j \leqs m/2$ such that $x$ is conjugate to an element of the form $(1, \ldots, 1, z_{2j+1}, \ldots, z_m)\s$, where $\s = (1,2) \cdots (2j-1,2j) \in S_m$ and $z_i^2 = 1$ for all $i$. If $\ell \geqs 0$ denotes the number of nontrivial $z_i$, then $x$ is of type $a_{2j}$ if $\ell = 0$, type $b_{2j+\ell}$ if $\ell$ is odd and type $c_{2j+\ell}$ if $\ell \geqs 2$ is even. 

First assume $x \in G$ is an involution of type $a_h$, so $h$ is even and $|x^G|>\frac{1}{2}q^{h(n-h)}$ (see \cite[Proposition 3.22]{fpr2}). From the above description of the involutions in $\bar{N}$ we deduce that $x$ is $\bar{G}$-conjugate to $(1, \ldots, 1)\s \in \bar{N}$, where $\s = (1,2)\cdots (h-1,h) \in S_m$. Now
\[
|\s^{W}| = 2^{h/2}\frac{m!}{(h/2)!(m-h)!2^{h/2}} = \frac{m!}{(h/2)!(m-h)!}
\]
and by arguing as in the proof of Lemma \ref{l:sp0} we deduce that  
\[
|x^G \cap N| \leqs \frac{m!}{(h/2)!(m-h)!}(q+1)^{h/2} 
\]
(this is equality if $T = (C_{q+1})^m$). It is straightforward to verify the bound $\a(x) \leqs (1-t)/2$.

Now assume $x$ is of type $b_h$ or $c_h$, according to the parity of $h$. The case $h=1$ requires special attention. Here $|x^G| = q^n-1$ and $x$ is $\bar{G}$-conjugate to $(z,1,\ldots, 1) \in \bar{L}^m$, which implies that $|x^G \cap N| \leqs m(q+1)$ and we deduce that $\a(x) \leqs (1-t)/2$. Now assume $h \geqs 2$, so $|x^G| > \frac{1}{2}q^{h(n-h+1)}$. Here there exists an integer $j$ in the range $0 \leqs j < h/2$ such that $x$ is $\bar{G}$-conjugate to an element in $\bar{N}$ of the form $(z_1, \ldots, z_m)\s$, where $\s = (1,2)\cdots (2j-1,2j) \in S_m$ and $z_i$ is nontrivial (and equal to $z$) if and only if $2j+1\leqs i \leqs h$. As a consequence, we deduce that
\begin{equation}\label{e:sym1}
|x^G \cap N| \leqs \sum_{j=0}^{\lceil h/2 \rceil-1} \frac{m!}{j!(m-2j)!2^j}2^j(q+1)^j \binom{m-2j}{h-2j}(q+1)^{h-2j}
\end{equation}
(once again, this is equality if $T = (C_{q+1})^m$) and thus 
\[
|x^G \cap N| \leqs \frac{m!}{(m-h)!}(q+1)^hf(h),
\]
where
\[
f(h) = \sum_{j=0}^{\lceil h/2 \rceil-1} \frac{1}{j!(h-2j)!3^j}.
\]
One can check that $f(h)$ is a decreasing function, so $f(h) \leqs f(2) = 1/2$ and we deduce that
\begin{equation}\label{e:sym2}
|x^G \cap N| \leqs \frac{m!}{(m-h)!2}(q+1)^h.
\end{equation}
By considering the bound in \eqref{e:sym2}, we may assume $n \leqs 20$ and $q = 2$. In each of these cases, one can check that the upper bound in \eqref{e:sym1} is sufficient unless $(n,q) = (10,2)$ and $x = b_3$. In the latter case, the more accurate  bounds $|x^G|>2^{24}$ and 
\[
|x^G \cap N| \leqs \binom{5}{3}3^3 + 2\cdot \frac{5!}{3!2}3^3 = 810
\]
are sufficient.

\vs

\noindent \emph{Case 2. $r \ne p$, $r=2$.}

\vs

Now suppose $x$ is a semisimple involution, so $q$ is odd. First assume $x$ lifts to an element of order $4$ in ${\rm Sp}_n(q)$, in which case 
\[
|x^G| = \frac{|{\rm Sp}_n(q)|}{2|{\rm GL}_m^{\e}(q)|} > \frac{1}{4}\left(\frac{q}{q+1}\right)q^{m(m+1)},
\]
where $q \equiv \e \imod{4}$. Here the trivial bound
\[
|x^G \cap N| \leqs |N| \leqs (q+1)^m|W|
\]
is sufficient if $n \geqs 18$ or $q \geqs 81$, so we may assume $n \leqs 16$ and $q \leqs 79$. To handle these cases, we can work with the more accurate bound
\[
|x^G \cap N| \leqs i_2(N) \leqs (q+1)^m(1+i_2(W)),
\]
where 
\[
1+i_2(W) = \sum_{j=0}^{\lfloor m/2 \rfloor}\frac{m!}{j!(m-2j)!}2^{m-2j}.
\]
One can check that the given bounds on $|x^G \cap N|$ and $|x^G|$ are sufficient. 

For the remainder, let us assume $x$ lifts to an involution in ${\rm Sp}_n(q)$ of the form $(-I_{2\ell}, I_{n-2\ell})$ for some $1 \leqs \ell \leqs \lfloor m/2 \rfloor$. Then for some integer $0 \leqs j \leqs \ell$, $x$ is $\bar{G}$-conjugate to an element in $\bar{N}$ of the form $(z_1, \ldots, z_m)\s$, where $\s = (1,2) \cdots (2j-1,2j) \in S_m$ and $z_i$ is nontrivial (and equal to $-I_2$) if and only if $2j+1 \leqs i \leqs \ell+j$. This implies that
\[
|x^G \cap N| \leqs \sum_{j=0}^{\ell} \frac{m!}{j!(m-2j)!2^j}2^j(q+1)^j\binom{m-2j}{\ell-j}, 
\]
which in turn yields 
\begin{equation}\label{e:spp2}
|x^G \cap N| \leqs  \frac{m!}{(m-2\ell)!}(q+1)^{\ell}\left(\sum_{j=0}^{\ell}\frac{1}{j!(\ell-j)!}\right) \leqs \frac{m!}{(m-2\ell)!}2(q+1)^{\ell}.
\end{equation}
Now $|x^G|>\frac{1}{2a}q^{2\ell(n-2\ell)}$, where $a=2$ if $\ell=m/2$, otherwise $a=1$, and one can check that this bound with \eqref{e:spp2} is always sufficient.

\vs

\noindent \emph{Case 3. $r>2$, $(r,|T|)=1$.}

\vs

Here $x$ is $\bar{G}$-conjugate to an element of the form $(1, \ldots, 1)\s \in \bar{N}$, where $\s \in S_m$ has cycle-shape $(r^{j},1^{m-jr})$ for some integer $j$ in the range $1 \leqs j \leqs \lfloor m/r \rfloor$. If $r=p$ then $x$ has Jordan form 
$(J_r^{2j},J_1^{n-2j r})$ on $V$, so  
\[
|x^G| = \frac{|{\rm Sp}_{n}(q)|}{q^{j(2n-2j-2jr)}|{\rm Sp}_{2j}(q)||{\rm Sp}_{n-2j r}(q)|} > \frac{1}{2}q^{j(r-1)(2n-2j r+1)}
\]
and by arguing as in the proof of Lemma \ref{l:sp0} we deduce that 
\begin{equation}\label{e:spp3}
|x^G \cap N| \leqs 2^{j(r-1)}\frac{m!}{j!(m-jr)!r^{j}}(q+1)^{j(r-1)}.
\end{equation}
Similarly, if $r \ne p$ then $x$ has Jordan form $(I_{n-2j(r-1)},\omega I_{2j}, \ldots, \omega^{r-1}I_{2j})$ on the natural module for $\bar{G}$, so 
\[
|x^G| > \frac{1}{2}\left(\frac{q}{q+1}\right)^{(r-1)/2}q^{j(r-1)(2n-2jr+1)}
\]
and \eqref{e:spp3} holds. In all cases, one can check that these bounds are sufficient.

\vs

\noindent \emph{Case 4. $r>2$, $r$ divides $|T|$.}

\vs

To complete the proof, we may assume $x$ is semisimple and $r$ is an odd prime divisor of $|T|$. Let $0 \leqs \ell \leqs \lfloor m/r \rfloor$ be maximal such that $x$ is $\bar{G}$-conjugate to an element in a coset $\bar{T}\pi$, where $\pi \in W$ is of the form $(1, \ldots, 1)\s$ and $\s \in S_m$ has cycle-shape $(r^{\ell},1^{m-r\ell})$. 

First assume $\ell=0$, so $x^G \cap N \subseteq T$. Set $s = \nu(x)$ (see Definition \ref{d:nu}). If $s=2$ then $x$ is $\bar{G}$-conjugate to an element of the form $(I_{n-2}, \omega, \omega^{-1})$, so $|x^G \cap N| \leqs 2m$ and the bound $|x^G|>\frac{1}{2}(q+1)^{-1}q^{2n-1}$ is sufficient. Next  assume $n=6$ and $s \geqs 3$. Here $|x^G|>\frac{1}{2}(q+1)^{-1}q^{13}$ and one can check that the bound $|x^G \cap N| \leqs (q+1)^3$ is good enough (recall that we may assume $q \geqs 7$; see \eqref{e:sp}). Finally, if $n \geqs 8$ and $s \geqs 4$ then $|x^G|>\frac{1}{2}(q+1)^{-1}q^{4n-15}$ and the trivial bound $|x^G \cap N| \leqs (q+1)^m$ is sufficient.

Now suppose $\ell \geqs 1$, so $r \leqs m$ and each $r$-th root of unity has multiplicity at least $2\ell$ as an eigenvalue of $x$ on $\bar{V}$. Now $|x^G|$ is minimal when $x$ is $\bar{G}$-conjugate to an element of the form $(I_{n-2\ell(r-1)},\omega I_{2\ell}, \ldots, \omega^{r-1}I_{2\ell})$, which implies that 
\[
|x^G|>\frac{1}{2}\left(\frac{q}{q+1}\right)^{(r-1)/2}q^{\ell(r-1)(2n-2\ell r+1)}.
\]
Next observe that there exists an integer $j$ in the range $0 \leqs j \leqs \ell$ such that $x$ is $\bar{G}$-conjugate to an element of the form $(z_1, \ldots, z_m)\rho \in \bar{N}$, where 
\[
\rho = (1, \ldots, r) (r+1, \ldots, 2r) \cdots ((j-1)r+1, \ldots, jr) \in S_m
\]
has cycle-shape $(r^j,1^{m-rj})$ and $z_i = {\rm diag}(\l_i,\l_i^{-1})$ is nontrivial (of order $r$) only if $i>jr$. Since there are at most $r$ possibilities for each $z_i$ with $i>jr$, we deduce that  
\begin{equation}\label{e:sp2}
|x^G \cap N| \leqs \sum_{j=0}^{\ell} 2^{j(r-1)}\frac{m!}{j!(m-jr)!r^j}(q+1)^{j(r-1)}r^{m-jr},
\end{equation}
which in turn implies that
\begin{equation}\label{e:sp3}
|x^G \cap N| \leqs (2(q+1))^{\ell(r-1)}r^m \sum_{j=0}^{\ell} \left(\frac{m^r}{r}\right)^j
\leqs 2\left(\frac{m^r}{r}\right)^{\ell}\left(2(q+1)\right)^{\ell(r-1)}r^m.
\end{equation}

Suppose $q \geqs 3$. Here one checks that the upper bound on $|x^G \cap N|$ in \eqref{e:sp3} is sufficient if $n \geqs 18$ or $q \geqs 23$, so we may assume $n \leqs 16$ and $q \leqs 19$. In the remaining cases, we find that the more accurate upper bound in \eqref{e:sp2} is sufficient. Similarly, if $q=2$ then the bound in \eqref{e:sp3} is good enough for $n \geqs 40$ and we see that \eqref{e:sp2} is sufficient if $12 \leqs n \leqs 38$. Finally, suppose $(n,q) = (10,2)$. Here \eqref{e:sp3} is sufficient unless $(r,\ell) = (3,1)$, which means that $x$ is of the form $(I_6,\omega I_2, \omega^2I_2)$ or $(I_4,\omega I_3, \omega^2I_3)$. In the latter case, $|x^G|>2^{35}$ and the bound in \eqref{e:sp3} is good enough. On the other hand, if $x = (I_6,\omega I_2, \omega^2I_2)$ then $|x^G|>2^{29}$ and the result follows since
\[
|x^G \cap N| \leqs \binom{5}{2}2^2 + 2^2\cdot \frac{5!}{2!3}3^2 = 760.
\]
\end{proof}

The following result completes the proof of Theorem \ref{t:main} for symplectic groups.
 
\begin{prop}\label{l:psp4}
The conclusion to Theorem \ref{t:main} holds if $G = {\rm PSp}_{4}(q)'$.
\end{prop}

\begin{proof}
Set $\what{G} = {\rm Sp}_4(q)$ and $Z = Z(\what{G})$, so we may write $G = \what{G}/Z$, $T = \what{T}/Z$ and $N = \what{N}/Z$, where $\what{N} = N_{\what{G}}(\bar{T})$. The Weyl group $W = S_2 \wr S_2 = D_8$ has $5$ conjugacy classes and up to conjugacy we see that $\what{T}$ is one of the following:
\[
(C_{q-\e})^2, \; C_{q+1}\times C_{q-1},\; C_{q^2-\e}
\]
with $\e=\pm$ and thus $N = [(q-\e)^2/d].D_8$, $[(q^2-1)/d].2^2$ or $[(q^2+1)/d].4$,  where $d=(2,q-1)$. The groups with $q \leqs 8$ can be handled using {\sc Magma}, so we will assume $q \geqs 9$.

We claim that $\what{\mathcal{Q}}(G,N,2)<1$ and thus $b(G,N) =2$. Let $x \in N$ be an element of prime order $r$. If $\dim x^{\bar{G}} \geqs 6$ then we have the trivial bound $|x^G \cap N| \leqs |N| \leqs 8(q+1)^2/d=a_1$ and $|x^G| \geqs b_1$, where
\[
b_1 = \left\{\begin{array}{ll}
\frac{1}{2}q(q-1)(q^4-1) & \mbox{$q$ odd} \\
q^3(q-1)(q^2+1) & \mbox{$q$ even.}
\end{array}\right.
\]
Therefore, the contribution to $\what{\mathcal{Q}}(G,N,2)$ from these elements is less than $a_1^2/b_1 < 2/3$. For the remainder, we may assume $\dim x^{\bar{G}} = 4$, so $x$ is an involution and either $q$ is odd and $x = (-I_2,I_2)$, or $q$ is even and $x$ is a long or short root element. 

First assume $q$ is odd, so 
\[
|x^G| = \frac{|{\rm Sp}_4(q)|}{2|{\rm Sp}_2(q)|^2} = \frac{1}{2}q^2(q^2+1)=b_2
\]
and $|x^G \cap N| = (i_2(\what{N})-1)/2$. We claim that $i_2(\what{N}) \leqs 2q+5$ and thus $|x^G \cap N| \leqs q+2 = a_2$.
To see this, first assume $\what{T} = (C_{q-\e})^2$ and note that $\what{N} = Q_{2(q-\e)} \wr S_2 < {\rm Sp}_{2}(q) \wr S_2$, where $Q_{2(q-\e)}$ is the generalised quaternion group of order $2(q-\e)$. Here $i_2(\what{N}) = 3+2(q-\e) \leqs 2q+5$ as required. Similarly, if $\what{T} = C_{q+1} \times C_{q-1}$ then $\what{N} = Q_{2(q+1)} \times Q_{2(q-1)}$ and $i_2(\what{N})=3$. 

Finally, suppose $\what{T} = C_{q^2-\e}$, so $\what{T} < {\rm Sp}_{2}(q^2)$ and $\what{N} = N_{{\rm Sp}_2(q^2)}(\what{T}).\la \phi \ra = Q_{2(q^2-\e)}.\la \phi \ra$, where $\phi$ is an involutory field automorphism of ${\rm Sp}_2(q^2)$. Since $Q_{2(q^2-\e)}$ has a unique involution, it remains to show that there are at most $2q+4$ involutions in the coset $Q_{2(q^2-\e)}\phi$. First assume $\e=+$ and write $\mathbb{F}_{q^2}^{\times} = \la \l \ra$. Then by taking $\what{T}$ to be the diagonal matrices in ${\rm Sp}_{2}(q^2)$ we get
\[
\what{N} = \left\la {\rm diag}(\l,\l^{-1}), \tau, \phi \right\ra, \;\; \mbox{with } \tau = \left(\begin{smallmatrix} \phantom{-}0 & 1 \\ -1 & 0 \end{smallmatrix}\right) \in {\rm Sp}_{2}(q^2)
\]
and it is straightforward to show that $i_2(\what{N}) = 2q+1$. On the other hand, if $\e=-$ then $\what{N} = C_{q^2+1}.4$ and every involution in $\what{N}$ is contained in $C_{q^2+1}.2 = Q_{2(q^2+1)}$. Therefore $i_2(\what{N}) = 1$ and thus $|x^G \cap N| = 0$.  

This justifies the claim and we conclude that 
\[
\what{\mathcal{Q}}(G,N,2) < a_1^2/b_1 + a_2^2/b_2<1
\]
if $q \geqs 9$ is odd.

Finally, let us assume $q \geqs 16$ is even and $x$ is a long or short root element, so $|x^G| = q^4-1 = b_3$. We claim that $|x^G \cap N| \leqs 2q+2 = a_3$. In order to establish the claim, let us first assume $T = (C_{q-\e})^2$, so $N = D_{2(q-\e)} \wr S_2$. The long root elements in $N$ correspond to the involutions in each $D_{2(q-\e)}$ factor, so $|x^G \cap N| = 2(q-\e)$. Similarly, the short root elements in $N$ are the involutions outside $(D_{2(q-\e)})^2$, so once again $|x^G \cap N| = 2(q-\e)$. Next suppose $T = C_{q+1} \times C_{q-1}$. Here $N = D_{2(q+1)} \times D_{2(q-1)}$ and the long root elements correspond to the involutions in each factor, so $|x^G \cap N| = 2q$, and we note that $N$ does contain any short root elements. Finally, suppose $T = C_{q^2-\e}$, in which case $N = D_{2(q^2-\e)}.2$ with $D_{2(q^2-\e)} < {\rm Sp}_{2}(q^2)$ and we note that $N$ does not contain any long root elements. If $\e=-$ then $N = C_{q^2+1}.4$ and every involution in $N$ is contained in $D_{2(q^2+1)}$, so $N$ does not contain any root elements at all in this case. On the other hand, if $\e=+$ then 
\[
N = \left\la {\rm diag}(\l,\l^{-1}), \tau, \phi \right\ra, \;\; \mbox{with } \tau = \left(\begin{smallmatrix} 0 & 1 \\ 1 & 0 \end{smallmatrix}\right) \in {\rm Sp}_{2}(q^2),
\]
where $\phi$ is an involutory field automorphism of ${\rm Sp}_{2}(q^2)$. It is straightforward to check that the coset $D_{2(q^2-1)}\phi$ contains $q+1$ involutions. 

For $q \geqs 16$ even, we conclude that 
\[
\what{\mathcal{Q}}(G,N,2) < a_1^2/b_1 + 2a_3^2/b_3<1
\]
and this completes the proof of the proposition. 
\end{proof}

\section{Odd dimensional orthogonal groups}\label{s:o_odd}

In order to complete the proof of Theorem \ref{t:main} for classical groups, we may assume $G$ is an orthogonal group. In this section we assume $G = \O_n(q)$, where $n \geqs 7$ is odd; the even dimensional orthogonal groups will be handled in Section \ref{s:o_even}.

Write $n=2m+1$ and note that $q$ is odd. Fix an $F$-stable maximal torus $\bar{T}$ of $\bar{G}$ and set $\bar{N} = N_{\bar{G}}(\bar{T}) = \bar{T}.W$, where $W = S_2 \wr S_m$ is the hyperoctahedral group. We may assume $\bar{T}$ is the group of diagonal matrices in $\bar{G} = {\rm SO}_n(k)$ of the form
\[
{\rm diag}(\l_1, \l_1^{-1}, \ldots, \l_m, \l_{m}^{-1},1) 
\]
with respect to a standard orthogonal basis $\{e_1,f_1, \ldots, e_m,f_m,v\}$ of the natural module $\bar{V}$. Note that $\bar{N}$ stabilises the orthogonal decomposition $\bar{V} = \bar{V}_1 \perp \cdots \perp \bar{V}_m \perp \la v \ra$, where $\bar{V}_i = \la e_i,f_i\ra$. More precisely, if we set 
\[
\bar{L} = \left\langle {\rm diag}(\l,\l^{-1}), z \,:\, \l \in k^{\times} \right\rangle  = {\rm O}_2(k), 
\]
where $z=\left(\begin{smallmatrix} 0 & 1 \\ 1 & 0 \end{smallmatrix}\right)$, then 
\[
\bar{N} = \left\{ {\rm diag}(A,1), \, {\rm diag}(B,-1)
\,:\, A,B \in \bar{L} \wr S_m, \,  \det(A) = 1, \, \det(B) = -1 \right\}.
\]
We will use the notation $(z_1, \ldots, z_m,y)\s$ to denote a general element of $\bar{N}$, where $z_i \in \bar{L}$, $y = \prod_i \det(z_i) \in \{-1,1\}$ and $\s \in S_m$. 

Recall from Section \ref{ss:tori} that there is a bijection from the set of conjugacy classes in $W$ to the set of $\bar{G}^F$-classes of $F$-stable maximal tori in $\bar{G}$. As recorded in the previous section, the conjugacy classes in $W$ are parameterised by pairs of partitions $(\l,\mu)$ with $|\l|+|\mu|=m$. Fix an element $w \in W$ corresponding to the pair $(\l,\mu)$, where $\l = (m^{a_m}, \ldots, 1^{a_1})$ and $\mu = (m^{b_m}, \ldots, 1^{b_1})$. Then the $W$-class of $w$ corresponds to an $F$-stable maximal torus $\bar{T}_w$ of $\bar{G}$ and we set $N = N_G(\bar{T}_w) = T.R$, where $R \cong C_W(w)$ and $T = \what{T} \cap G$ with 
\[
\what{T} = \left(C_{q^{m}-1}\right)^{a_m} \times \cdots \times \left(C_{q-1}\right)^{a_1} \times \left(C_{q^{m}+1}\right)^{b_m} \times \cdots \times \left(C_{q+1}\right)^{b_1} < {\rm SO}_n(q).
\]
We refer the reader to \cite[Theorem 4]{BuG} for the cyclic structure of $T$.

\begin{lem}\label{so_odd1}
Let $G = \O_n(q)$, where $n \geqs 7$ is odd.
\begin{itemize}\addtolength{\itemsep}{0.2\baselineskip}
\item[{\rm (i)}] If $n \leqs 11$ and $q \leqs 13$ then $\eta_G(t)<1$, where $t$ is defined as follows: 
\[
\begin{array}{r|cccccc}
 & q= 3 & 5 & 7 & 9 & 11 & 13 \\ \hline
n = 7 & 23/100 & 17/100 & 7/50 & 7/50 & 13/100 & 13/100 \\
9 & 17/100 & 13/100 & 11/100 & 1/10 & 11/100 & 11/100 \\
11 & 13/100 & 1/10 & 9/100 & 2/25 & 2/25 & 2/25
\end{array}
\]
\item[{\rm (ii)}] If $13 \leqs n \leqs 21$ and $q=3$ then $\eta_G(t)<1$ for $t = 11/100$.
\end{itemize}
\end{lem}

\begin{proof}
This is an entirely straightforward {\sc Magma} calculation, working with the standard matrix representation of $G$ over $\mathbb{F}_q$ and the function 
\texttt{ClassicalClasses} to compute the relevant conjugacy class sizes. 
\end{proof}

\begin{thm}\label{t:so_odd}
If $G = \O_{n}(q)$ with $n \geqs 7$, then $b(G,N) = 2$.
\end{thm}

\begin{proof}
The cases $(n,q) \in \{(7,3),(9,3)\}$ can be handled using {\sc Magma} (see Lemma \ref{l:small}), so we may (and will) assume $(n,q) \ne (7,3), (9,3)$ for the remainder. Set $t=1/3$, with the exception of the cases in Lemma \ref{so_odd1}, where we define $t$ as in the lemma. Let $x \in N$ be an element of prime order $r$.  As before, our aim is to show that $\a(x) \leqs (1-t)/2$. We will adopt the notation introduced at the start of Section \ref{s:o_odd}.

\vs

\noindent \emph{Case 1. $r=2$.}

\vs

Here $x$ is of the form $(-I_{2\ell},I_{n-2\ell})$ with $1 \leqs \ell \leqs m$, whence
\begin{equation}\label{e:soo}
|x^G|>\frac{1}{4}q^{2\ell(n-2\ell)}.
\end{equation}
The cases $\ell \in \{1,m-1,m\}$ require special attention. Set $z = \left(\begin{smallmatrix} 0 & 1 \\ 1 & 0 \end{smallmatrix}\right) \in {\rm O}_2(k)$.

First assume $\ell=m$. Here $x$ is $\bar{G}$-conjugate to elements in $\bar{N}$ of the form $(-I_2, \ldots, -I_2,1)$ and $(z,-I_2, \ldots, -I_2,-1)$, which implies that $|x^G \cap N| \leqs m(q+1)+1$. One can now check that the bound on $|x^G|$ in \eqref{e:soo} is sufficient. 

Next assume $\ell=1$. In this case, $x$ is $\bar{G}$-conjugate to the following elements in $\bar{N}$:
\[
(-I_2,I_2, \ldots, I_2,1),\, (z,I_2, \ldots, I_2,-1), \, (z,z,I_2, \ldots, I_2,1), \, (I_2, \ldots, I_2,1)\s, 
\]
where $\s = (1,2) \in S_m$. This implies that
\[
|x^G \cap N| \leqs m + m(q+1)+\binom{m}{2}(q+1)^2+2\binom{m}{2}(q+1)
\]
and once again the bound in \eqref{e:soo} is good enough. Similarly, if $\ell=m-1$ then $x$ is $\bar{G}$-conjugate to the following: 
\[
(I_2,-I_2, \ldots, -I_2,1),\, (z,I_2,-I_2, \ldots, -I_2,-1), \, (z,z,-I_2, \ldots, -I_2,1),
\]
\[
(z,z,z,-I_2, \ldots, -I_2,-1),\, (I_2,I_2,-I_2, \ldots, -I_2,1)\s, \, (I_2,I_2,z,-I_2, \ldots, -I_2,-1)\s
\]
where $\s=(1,2) \in S_m$. Therefore,
\begin{align*}
|x^G \cap N| \leqs  & \; m+2\binom{m}{2}(q+1) + \binom{m}{2}(q+1)^2 + \binom{m}{3}(q+1)^3 \\
& \, +2\binom{m}{2}(q+1) + 2\binom{m}{2}(m-2)(q+1)^2
\end{align*}
and it is routine to check that \eqref{e:soo} is sufficient.

To complete the analysis of involutions, we may assume $n \geqs 9$ and $2 \leqs \ell \leqs m-2$.  First observe that the combined contribution to $|x^G \cap N|$ from elements in cosets of the form $Ty$ with $y \in (S_2)^m < W$ is at most $(2(q+1))^m$. Similarly, the contribution from elements in cosets $Ty$ such that $y \in (S_2)^m\s$ and $\s \in S_m$ has cycle-shape $(2^j,1^{m-2j})$ with $1 \leqs j \leqs \min\{\ell, (n-2\ell-1)/2\}$ is at most 
\[
\frac{m!}{j!(m-2j)!2^j}(2(q+1))^j \cdot (2(q+1))^{m-2j}.
\]
Therefore, if we set $a = \min\{\ell, (n-2\ell-1)/2\}$, then
\begin{equation}\label{e:soo1}
|x^G \cap N| \leqs 2^m(q+1)^m\sum_{j=0}^a\frac{m!}{j!(m-2j)!2^j}. 
\end{equation}
Now
\[
\sum_{j=0}^a\frac{m!}{j!(m-2j)!2^j} \leqs \sum_{j=0}^a \left(\frac{m^2}{2}\right)^j < 2\left(\frac{m^2}{2}\right)^a
\]
and this implies that
\begin{equation}\label{e:soo2}
|x^G \cap N| \leqs 2^{m+1}(q+1)^m\left(\frac{m^2}{2}\right)^a.
\end{equation}
One can now check that the bounds in \eqref{e:soo2} and \eqref{e:soo} are sufficient if $n \geqs 39$ or $q \geqs 23$. And if $n \leqs 37$ and $q \leqs 19$, then the bounds in \eqref{e:soo1} and \eqref{e:soo} are good enough. 

\vs

\noindent \emph{Case 2. $r>2$, $(r,|T|)=1$.}

\vs

First assume $r=p$.  Here $x$ is $\bar{G}$-conjugate to an element of the form $(1, \ldots, 1)\s \in \bar{N}$, where $\s = (1, \ldots, 1)\rho \in W$ and $\rho \in S_m$ has cycle-shape $(r^j,1^{m-jr})$ for some $1 \leqs j \leqs \lfloor m/r \rfloor$. Then $x$ has Jordan form $(J_r^{2j},J_1^{n-2jr})$ on $V$ and we deduce that 
\begin{equation}\label{e:sounip}
|x^G \cap N| \leqs \frac{m!}{j!(m-jr)!r^{j}}\left(2(q+1)\right)^{j(r-1)}.
\end{equation}
In addition, we have
\[
|x^G|>\frac{1}{8}q^{j(r-1)(2n-2jr-1)}
\]
and one can check that these bounds yield $\a(x) \leqs (1-t)/2$. 

Similarly, if $r \ne p$ then $x$ is $\bar{G}$-conjugate to $(I_{n-2j(r-1)},\omega I_{2j}, \ldots, \omega^{r-1}I_{2j})$ for some integer $j$ in the range $1 \leqs j \leqs \lfloor m/r \rfloor$, so 
\[
|x^G|>\frac{1}{2}\left(\frac{q}{q+1}\right)^{(r-1)/2}q^{j(r-1)(2n-2jr-1)}
\]
and the upper bound on $|x^G \cap N|$ in \eqref{e:sounip} is satisfied. Once again, it is straightforward to show that these bounds are sufficient.

\vs

\noindent \emph{Case 3. $r>2$, $r$ divides $|T|$.}

\vs

Here we proceed as in Case 4 in the proof of Theorem \ref{t:sp}. Let $0 \leqs \ell \leqs \lfloor m/r\rfloor$ be maximal such that $x$ is $\bar{G}$-conjugate to an element in a coset $\bar{T}\pi$, where $\pi = (1,\ldots, 1)\s \in W$ and $\s \in S_m$ has cycle-shape $(r^{\ell},1^{m-r\ell})$. Set $s = \nu(x)$.

Suppose $\ell=0$, so $x^G \cap N \subseteq T$. If $s=2$ then $x = (I_{n-2},\omega,\omega^{-1})$ up to $\bar{G}$-conjugacy and one can check that the bounds $|x^G \cap N| \leqs 2m$ and $|x^G|>\frac{1}{2}(q+1)^{-1}q^{2n-3}$ are sufficient. Now assume $s \geqs 4$. If $n=7$ then $|x^G|>\frac{1}{2}(q+1)^{-1}q^{13}$ (minimal if $x = (I_1,\omega I_3, \omega^{-1}I_3)$) and the trivial bound $|x^G \cap N| \leqs (q+1)^3$ yields $\a(x) \leqs 1/3$. For $n \geqs 9$ we have $|x^G|>\frac{1}{2}(q+1)^{-1}q^{4n-13}$ and once again the bound $|x^G \cap N| \leqs (q+1)^m$ is good enough.

Finally, suppose $\ell \geqs 1$. Here $r \leqs m$ and each $r$-th root of unity has multiplicity at least $2\ell$ as an eigenvalue of $x$ on $\bar{V}$ (in particular, the $1$-eigenspace of $x$ is at least $3$-dimensional). Since $|x^G|$ is minimal when $x = (I_{n-2\ell(r-1)}, \omega I_{2\ell}, \ldots, \omega^{r-1}I_{2\ell})$, it follows that
\[
|x^G|>\frac{1}{2}\left(\frac{q}{q+1}\right)^{(r-1)/2}q^{\ell(r-1)(2n-2\ell r-1)}
\]
and the upper bounds on $|x^G \cap N|$ in \eqref{e:sp2} and \eqref{e:sp3} are satisfied. One can check that the bound in \eqref{e:sp3} is sufficient if $n \geqs 13$ or $q \geqs 23$; in each of the remaining cases, we can evaluate the bound in \eqref{e:sp2} and the result follows.
\end{proof}

\section{Even dimensional orthogonal groups}\label{s:o_even}

Here we complete the proof of Theorem \ref{t:main} for classical groups by handling the even dimensional orthogonal groups $G = {\rm P\O}_n^{\e}(q)$, where $n = 2m \geqs 8$ and $\e = \pm$.

Write $\bar{N} = \bar{T}.W$, where $\bar{T}$ is the image (modulo scalars) of the diagonal matrices in ${\rm SO}_n(k)$ of the form ${\rm diag}(\l_1,\l_1^{-1}, \ldots, \l_m,\l_m^{-1})$ with respect to a standard basis $\{e_1, f_1, \ldots,  e_m,f_m\}$ for the natural module $\bar{V}$. Set 
\[
\bar{L} = \la {\rm diag}(\l,\l^{-1}), z \,:\, \l \in k^{\times}\ra = {\rm O}_2(k),
\]
where $z = \left(\begin{smallmatrix} 0 & 1 \\ 1 & 0 \end{smallmatrix}\right)$. Then $\bar{N}$ is the image in $\bar{G}$ of the subgroup 
\[
\la \bar{T}, (z,z,1, \ldots, 1),S_m \ra < \bar{L} \wr S_m
\]
and thus $W = 2^{m-1}{:}S_m$ is an index-two subgroup of $S_2 \wr S_m$.  
We will use the notation $(z_1, \ldots, z_m)\s$ to denote a general element of $\bar{N}$, where $z_i \in \bar{L}$ and $\s \in S_m$. Without loss of generality, we may assume $\bar{T}$ is $F$-stable.

Recall that there is a bijection from the set of $F$-classes in $W$ to the set of $\bar{G}^F$-classes of $F$-stable maximal tori in $\bar{G}$. First assume $\e=+$. Here the $F$-classes in $W$ coincide with the usual conjugacy classes in $W$, which can be parameterised by pairs of partitions $(\l,\mu)$, where $|\l|+|\mu|=m$ and the number of parts in $\mu$ is even, with the additional condition that if every part of $\l$ is even and $\mu$ is the empty partition, then there are two $W$-classes corresponding to $(\l,\mu)$. Similarly, if $\e=-$ and $m \geqs 5$ is odd, then the $F$-classes in $W$ are in bijection with the usual classes and they have essentially the same parameterisation, the only difference being that each partition $\mu$ in the pair $(\l,\mu)$ should have an odd number of parts (in particular, $\mu$ is non-empty). On the other hand, if $\e = -$ and $m \geqs 4$ is even, then there is a distinction to be made between the $F$-classes in $W$ and the usual conjugacy classes (see Remark \ref{r:minus} for more details). But in any case, the $F$-classes are still parameterised by pairs of partitions $(\l,\mu)$ as above, where $\mu$ has an odd number of parts (for example, if $m=4$ then $W$ has $13$ conjugacy classes, but we find that there are only $9$ distinct $F$-classes when $\e=-$).

So in all cases we may associate the $F$-class of $w \in W$ with a pair of partitions $(\l,\mu)$ as described above, according to the type of $F$. Write 
\begin{equation}\label{e:part}
\l = (m^{a_m}, \ldots, 1^{a_1}), \, \mu = (m^{b_m}, \ldots, 1^{b_1})
\end{equation}
and let $\bar{T}_w$ be the corresponding $F$-stable maximal torus of $\bar{G}$. We can then define $N = N_G(\bar{T}_w) = T.R$, where $R = C_{W,F}(w)$ (see \eqref{e:CW}) and $T$ is the image (modulo scalars) of $\what{T} \cap \O_{n}^{\e}(q)$, where
\[
\what{T} = \left(C_{q^{m}-1}\right)^{a_m} \times \cdots \times \left(C_{q-1}\right)^{a_1} \times \left(C_{q^{m}+1}\right)^{b_m} \times \cdots \times \left(C_{q+1}\right)^{b_1}
\]
is a maximal torus of ${\rm SO}_n^{\e}(q)$. Note that $C_{W,F}(w) \cong C_W(w)$ if $\e=+$, or if $\e=-$ and $m \geqs 5$ is odd. See \cite[Theorems 5-7]{BuG} for the precise cyclic structure of $T$.

\begin{rem}\label{r:minus}
Suppose $G = {\rm P\O}_{2m}^{-}(q)$ with $m \geqs 4$ even and set $N = T.R$ as above with respect to the pair of partitions $(\l,\mu)$ in \eqref{e:part}. Let $w \in W$ be a representative of the corresponding $F$-class in $W$. We can identify the $F$-centraliser $R = C_{W,F}(w)$ with an index-two subgroup of $C_{W_0}(w)$, where we view $W$ as an index-two subgroup of the Weyl group $W_0 = S_2 \wr S_m$ of ${\rm SO}_{2m+1}(k)$. Now $C_{W_0}(w) = A \times B$, where 
\begin{align*}
A & = ((C_{2} \times C_{m}) \wr S_{a_m}) \times ((C_{2} \times C_{m-1}) \wr S_{a_{m-1}}) \times \cdots \times (C_2 \wr S_{a_1}) \\
B & = (C_{2m} \wr S_{b_m}) \times (C_{2(m-1)} \wr S_{b_{m-1}}) \times \cdots \times (C_2 \wr S_{b_1}) 
\end{align*}
and so it is straightforward to compute $|R|$. For example, if $G = \O_{8}^{-}(2)$ then each $F$-class in $W$ corresponds to one of the pairs $(\l,\mu)$ in Table \ref{tab:mu}. In the table, we also describe the structure of $T$ (see \cite[Theorem 7]{BuG}) and $C_{W_0}(w)$, and we compute $|N|$.
\end{rem}

\begin{table}
\[
\begin{array}{ccllcc} \hline
\l & \mu & T & C_{W_0}(w) & |R| & |N| \\ \hline
(3) & (1) & C_{21} & C_6 \times C_2 & 6 & 126\\
(2,1) & (1) & (C_3)^2 & (C_2)^4 & 8 & 72 \\
(1^3) & (1) & C_3 & (C_2 \wr S_3) \times C_2 & 48 & 144 \\
(2) & (2) & C_{15} & C_4 \times (C_2)^2 & 8 & 120 \\
(1^2) & (2) & C_5 & D_8 \times C_4 & 16 & 80 \\
(1) & (3) & C_9 & C_6 \times C_2 & 6 & 54 \\
(1) & (1^3) & (C_3)^3 & (C_2 \wr S_3) \times C_2 & 48 & 1296 \\
\emptyset & (4) & C_{17} & C_8 & 4 & 68 \\
\emptyset & (2,1^2) & C_{15} \times C_3 & D_8 \times C_4 & 16 & 720 \\ \hline
\end{array}
\]
\caption{The case $G = \O_{8}^{-}(2)$}
\label{tab:mu}
\end{table}
 
\begin{lem}\label{so_even1}
Let $G = {\rm P\O}_n^{\e}(q)$, where $n \geqs 8$ is even.
\begin{itemize}\addtolength{\itemsep}{0.2\baselineskip}
\item[{\rm (i)}] If $n=8$ and $q \leqs 16$ then $\eta_G(t)<1$, where $t$ is defined as follows: 
\[
\begin{array}{r|cccccccccc}
q  & 2 & 3 & 4 & 5 & 7 & 8 & 9 & 11 & 13 & 16 \\ \hline
t & 7/25 & 11/50 & 9/50 & 4/25 & 7/50 & 13/100 & 13/100 & 7/50 & 7/50 & 3/20
\end{array}
\]
\item[{\rm (ii)}] If $(n,q)$ is one of the following, then $\eta_G(t)<1$ where $t$ is defined as in the table: 
\[
\begin{array}{r|cccccc}
  & n = 10 & 12 & 14 & 16 & 18 & 20 \\ \hline
q=2 & 9/50 & 7/50 & 11/100 & 1/10 & 2/25 & 7/100 \\
3 & 13/100 & 11/100 & 9/100 & 2/25 & 7/100 & 3/50 \\
4 & 3/25 & 9/100 & & & & \\
5 & 11/100 & 9/100 & & & & \\
7 & 1/10 & 7/100 & & & & 
\end{array}
\]
\end{itemize}
\end{lem}

\begin{proof}
This is very similar to the proof of Lemma \ref{l:sp1}. With the aid of {\sc Magma}, we work with the standard matrix representation of $L = \O_n^{\e}(q)$ and we use the function \texttt{ClassicalClasses} to compute the relevant class lengths in $L$. In this way, with an appropriate adjustment for involutions when $q$ is odd and $Z(L)$ is nontrivial, we obtain the list of class sizes of elements of prime order in $G$. Note that the sizes of the classes of involutions in $G$ can be read off from \cite[Tables B.10, B.11]{BG} when $q$ is odd.
\end{proof}

\begin{thm}\label{t:so_even}
Suppose $G = {\rm P\O}_{n}^{\e}(q)$ and $n \geqs 8$ is even. Then $b(G,N) \leqs 3$, with equality if and only if $G = \O_{8}^{+}(2)$ and $N = 3^4{:}(2^3{:}S_4)$, in which case $|N \cap N^x| = 4$ for some $x \in G$.
\end{thm}

\begin{proof}
In view of Lemma \ref{l:small} we may assume $(n,q) \not\in \{ (8,2), (8,3), (10,2)\}$. Let $x \in N$ be an element of prime order $r$. As before, we seek to establish the bound $\a(x) \leqs (1-t)/2$, where either $t = 1/3$, or $G$ is one of the groups in Lemma \ref{so_even1} and we define $t$ as in the lemma. There are several cases to consider.

\vs

\noindent \emph{Case 1. $r=p=2$.}

\vs

Here $x$ has Jordan form $(J_2^h,J_1^{n-2h})$ on $V$ for some even integer $h = 2\ell \geqs 2$, whence $x$ is of type $a_h$ or $c_h$ with respect to the notation in \cite{AS}. Note that every involution in $\bar{L}^m \cap \bar{G}$ is of type $c$. 

If $x = a_h$ then $x$ is $\bar{G}$-conjugate to $(1, \ldots, 1)\s \in \bar{N}$ with $\s = (1,2) \cdots (h-1,h) \in S_m$, so 
\[
|x^G \cap N| \leqs 2^{\ell}\frac{m!}{\ell!(m-2\ell)!2^{\ell}}(q+1)^{\ell}
\]
and it is straightforward to check that the bound $|x^G|>\frac{1}{2}q^{h(n-h-1)}$ is sufficient. 

Now suppose $x = c_h$. If $h=2$ then $|x^G|>\frac{1}{2}q^{2(n-2)}$ and we observe that $x$ is $\bar{G}$-conjugate to $(z,z,1, \ldots, 1) \in \bar{N}$, which implies that $|x^G \cap N| \leqs \binom{m}{2}(q+1)^2$. These bounds are sufficient. Now assume $h \geqs 4$. Here there exists an integer $j$ in the range $0 \leqs j <h/2$ such that $x$ is $\bar{G}$-conjugate to an element in $\bar{N}$ of the form $(z_1, \ldots, z_m)\s$, where $\s = (1,2) \cdots (2j-1,2j) \in S_m$ and $z_i$ is nontrivial (and equal to $z$) if and only if $2j+1 \leqs i \leqs h$. Therefore,
\begin{equation}\label{e:oo1}
|x^G \cap N| \leqs \sum_{j=0}^{h/2-1} 2^j\frac{m!}{j!(m-2j)!2^j}(q+1)^j\binom{m-2j}{h-2j}(q+1)^{h-2j}
\end{equation} 
and we deduce that
\begin{equation}\label{e:oo2}
|x^G \cap N| \leqs \frac{m!}{(m-h)!}(q+1)^h\sum_{j=0}^{h/2-1} \frac{1}{j!(h-2j)!3^j} \leqs \frac{m!}{(m-h)!4}(q+1)^h.
\end{equation} 
Since $|x^G|>\frac{1}{2}q^{h(n-h)}$, one can check that the bound in \eqref{e:oo2} is sufficient unless $(n,q) = (12,2)$. In this case, we can evaluate the bound in \eqref{e:oo1} and the result follows.

\vs

\noindent \emph{Case 2. $r=2$, $p \ne 2$.}

\vs

Now assume $x \in N$ is a semisimple involution. We begin by considering the special case where $x$ lifts to an element of order $4$ in $\O_n^{\e}(q)$ (that is, $x = Z\hat{x}$ where $Z = Z(\O_n^{\e}(q))$, $\hat{x} \in \O_n^{\e}(q)$ and $\hat{x}^2 = -I_n$). Here 
\[
|x^G| >\frac{1}{4}\left(\frac{q}{q+1}\right)q^{n(n-2)/4}
\]
and one can check that the trivial bound $|x^G \cap N| \leqs |N| \leqs (q+1)^m|W|$ is sufficient unless $n \in \{8,10\}$, or $n \in \{12,14,16\}$ and $q \leqs 13$. To handle the remaining cases, we can work with the more accurate estimate
\[
|x^G \cap N| \leqs \sum_{j=0}^{\lfloor m/2 \rfloor}2^{j}(q+1)^j \cdot 2^{m-2j}\frac{m!}{j!(m-2j)!2^j},
\]
which is obtained by carefully considering the relevant elements of order $4$ in the normaliser of a maximal torus of ${\rm SO}_{n}(k)$. Indeed, working modulo scalars, we observe that $x$ is $\bar{G}$-conjugate to an element of the form $(z_1, \ldots, z_m)\pi \in \bar{N}$, where $\pi = (y_1, \ldots, y_m)\s \in W$, $\s = (1,2) \cdots (2j-1,2j) \in S_m$ for some $0 \leqs j \leqs \lfloor m/2 \rfloor$, $z_i = {\rm diag}(\l_i,\l_i^{-1})$ has order $4$ for all $i$, $y_i = 1$ if $i > 2j$ and $y_{2\ell-1}=y_{2\ell} \in \{1,z\}$ if $1 \leqs \ell \leqs j$. By combining this with the above lower bound on $|x^G|$, we deduce that $\a(x) \leqs (1-t)/2$ as required.

To complete the analysis of involutions, we may assume $x$ is $\bar{G}$-conjugate to an element of the form $(-I_{2\ell},I_{n-2\ell})$, where $1\leqs \ell \leqs \lfloor m/2 \rfloor$. First observe that
\begin{equation}\label{e:so7}
|x^G| >\frac{1}{4d}\left(\frac{q}{q+1}\right)q^{2\ell(n-2\ell)},
\end{equation}
where $d = 2$ if $\ell = m/2$, otherwise $d=1$. The cases with $\ell \in \{1,2\}$ require special attention.

Suppose $\ell=1$. Then $x$ is $\bar{G}$-conjugate to the following elements in $\bar{N}$:
\[
(-I_2,I_2, \ldots, I_2),\, (z,z,I_2, \ldots, I_2), \, (I_2, \ldots, I_2)\s,
\]
where $\s = (1,2) \in S_m$. This implies that
\[
|x^G \cap N| \leqs m+\binom{m}{2}(q+1)^2 + \binom{m}{2}2(q+1)
\]
and one can check that the bound in \eqref{e:so7} is sufficient.

Now assume $\ell=2$. Here $x$ is $\bar{G}$-conjugate to the following elements in $\bar{N}$: 
\[
(-I_2, -I_2, I_2, \ldots, I_2), \, (-I_2, z,z, I_2, \ldots, I_2), \, (z,z,z,z,I_2, \ldots, I_2),
\]
\[
(I_2,I_2,-I_2,I_2, \ldots, I_2)\s, \, (I_2,I_2,z,z,I_2, \ldots, I_2)\s, \, (I_2, \ldots, I_2)\rho,
\]
where $\s = (1,2) \in S_m$ and $\rho= (1,2)(3,4) \in S_m$. As a consequence, we deduce that
\begin{align*}
|x^G \cap N| \leqs & \, \binom{m}{2} + m\binom{m-1}{2}(q+1)^2 + \binom{m}{4}(q+1)^4 + \binom{m}{2}2(q+1)(m-2) \\
 & \, +  \binom{m}{2}2(q+1)\binom{m-2}{2}(q+1)^2 +  \frac{m!}{2!(m-4)!2^2}(2(q+1))^2
\end{align*}
and once again the result follows via the bound in \eqref{e:so7}.

Finally, let us assume $3 \leqs \ell \leqs \lfloor m/2 \rfloor$ and note that $m \geqs 6$. Here we can repeat the argument in the final paragraph of Case 1 in the proof of Theorem \ref{t:so_odd} to show that \eqref{e:soo1} and \eqref{e:soo2} hold (with $a=\ell$). 
It is straightforward to check that \eqref{e:soo2} is sufficient if $n \geqs 24$ or $q \geqs 13$, so we may assume $12 \leqs n \leqs 22$ and $q \leqs 11$. In these cases, we can evaluate the upper bound on $|x^G \cap N|$ in \eqref{e:soo1}, which is sufficient unless $(n,q) = (12,3)$. In the latter case, one can check that the trivial upper bound $|x^G \cap N| \leqs (q+1)^{m}(1+i_2(W))$ is good enough.

\vs

\noindent \emph{Case 3. $r>2$.}

\vs

If $r$ does not divide $|T|$ then we can proceed as in Case 2 in the proof of Theorem \ref{t:so_odd}; the argument goes through essentially unchanged and we omit the details. 

Finally, let us assume $r$ is an odd prime divisor of $|T|$. Here the analysis is  similar to Case 3 in the proof of Theorem \ref{t:so_odd}, but one or two cases require special attention and so we give the details. 

Let $0 \leqs \ell \leqs \lfloor m/r\rfloor$ be maximal such that $x$ is $\bar{G}$-conjugate to an element in a coset $\bar{T}\pi$, where $\pi = (1,\ldots, 1)\s \in W$ and $\s \in S_m$ has cycle-shape $(r^{\ell},1^{m-r\ell})$. Write $s = \nu(x)$.

First assume $\ell=0$, so $x^G \cap N \subseteq T$. If $s=2$ then $x$ is $\bar{G}$-conjugate to $(I_{n-2},\omega,\omega^{-1})$ and the bounds $|x^G \cap N| \leqs 2m$ and $|x^G|>\frac{1}{2}(q+1)^{-1}q^{2n-3}$ are sufficient. Now assume $s \geqs 4$. If $n \geqs 10$ then $|x^G|>\frac{1}{2}(q+1)^{-1}q^{4n-15}$ and the trivial bound $|x^G \cap N| \leqs (q+1)^m$ is good enough. If $n = 8$ and $x$ is not of the form $(\omega I_4, \omega^{-1}I_4)$, then $|x^G|>\frac{1}{2}(q+1)^{-1}q^{19}$ and the result follows since $|x^G \cap N| \leqs (q+1)^4$. On other hand, if $n=8$ and $x$ is conjugate to $(\omega I_4, \omega^{-1}I_4)$, then $|x^G|>\frac{1}{2}(q+1)^{-1}q^{13}$ and we note that $|x^G \cap N| \leqs 2^4$, which yields $\a(x) \leqs (1-t)/2$ as required.

For the remainder, let us assume $\ell \geqs 1$, so $r \leqs m$ and each $r$-th root of unity has multiplicity at least $2\ell$ as an eigenvalue of $x$ on $\bar{V}$. The upper bounds on $|x^G \cap N|$ in \eqref{e:sp2} and \eqref{e:sp3} are satisfied and we note that $|x^G|$ is minimal when $x = (I_{n-2\ell(r-1)}, \omega I_{2\ell}, \ldots, \omega^{r-1}I_{2\ell})$, which implies that 
\[
|x^G|>\frac{1}{2}\left(\frac{q}{q+1}\right)^{(r+1)/2}q^{\ell(r-1)(2n-2\ell r-1)}.
\]
For $q \geqs 3$, one can check that the bound in \eqref{e:sp3} is sufficient if $n \geqs 18$ or $q \geqs 13$, while \eqref{e:sp2} is effective in each of the remaining cases. Similarly, if $q=2$ then the same bounds are sufficient unless $n \leqs 28$, $r=3$ and $\ell=1$. So to complete the proof, we may assume the latter conditions are satisfied, in which case $x$ is of the form $(I_{n-4},\omega I_2, \omega^{-1}I_2)$ or 
$(I_{n-6},\omega I_3, \omega^{-1}I_3)$. In the latter case, $|x^G|>\frac{1}{2}(q+1)^{-1}q^{6n-29}$ and the upper bound on $|x^G \cap N|$ in \eqref{e:sp2} is sufficient. Finally, if $x = (I_{n-4},\omega I_2, \omega^{-1}I_2)$ then
\[
|x^G \cap N| \leqs 2^2\binom{m}{2}+\frac{m!}{(m-3)!3}(2(q+1))^2,\;\; |x^G|>\frac{1}{2}\left(\frac{q}{q+1}\right)q^{4n-14}
\]
and the result follows.
\end{proof}

\vs

This completes the proof of Theorem \ref{t:main} for classical groups.

\section{Exceptional groups}\label{s:excep}

In this final section we complete the proof of Theorem \ref{t:main} by handling the exceptional groups. Let $G = O^{p'}(\bar{G}^F)$ be a simple exceptional group of Lie type over $\mathbb{F}_q$, where $q=p^f$ and $p$ is a prime. Note that $G_2(2)' \cong {\rm U}_3(3)$ and ${}^2G_2(3)' \cong {\rm L}_2(8)$, so we may assume $q \geqs 3,27$ if $G = G_2(q), {}^2G_2(q)$, respectively. We adopt the notation from Section \ref{ss:tori}. In particular we have $N = N_G(\bar{T}_w) = T.R$, where $T = G \cap \bar{T}_w^F$ and $R = C_{W,F}(w)$ for some element $w$ in the Weyl group $W = N_{\bar{G}}(\bar{T})/\bar{T}$, where $\bar{T}$ is an $F$-stable maximal torus of $\bar{G}$. Our aim is to show that $b(G,N) = 2$ in every case.

We begin by recalling the following result, which handles the special cases where $N$ is a maximal subgroup of $G$. Note that these cases are recorded in \cite[Table 5.2]{LSS}.

\begin{prop}\label{p:Nmax}
If $N$ is a maximal subgroup of $G$, then $b(G,N) = 2$.
\end{prop}

\begin{proof}
This is \cite[Proposition 4.2]{BTh_ep}.
\end{proof}

\begin{cor}\label{c:2b2}
If $G = {}^2B_2(q)$, then $b(G,N) = 2$.
\end{cor}

\begin{proof}
By \cite{Suz}, $N$ is a maximal subgroup of $G$ and so we may apply Proposition \ref{p:Nmax} (the original reference is \cite[Lemma 4.39]{BLS}).
\end{proof}

\begin{prop}\label{p:e8}
If $G = E_8(q)$, then $b(G,N) = 2$.      
\end{prop}

\begin{proof}
This is a straightforward application of Proposition \ref{p:base}. Let $x \in G$ be an element of prime order. If $x$ is a long root element, then $|x^G| > q^{58} = b_1$ and Lemma \ref{l:root} implies that there are at most $a_1 = 120(q+1)^8$
such elements in $N$ since the Weyl group $W= 2.{\rm O}_8^{+}(2)$ contains $120$ reflections. In all other cases, we have $|x^G| > q^{92} = b_2$ (see 
\cite[Proposition 2.11]{BTh_ep}) and we note that $|N| \leqs (q+1)^8|W| = a_2$. Therefore, by applying Lemma \ref{l:calc}, we deduce that 
\begin{equation}\label{e:Q2}
\what{\mathcal{Q}}(G,N,2) < a_1^2/b_1 + a_2^2/b_2 < 1 
\end{equation}
and the result follows.
\end{proof}

\begin{prop}\label{p:e7}
If $G = E_7(q)$, then $b(G,N) = 2$.      
\end{prop}

\begin{proof}
For $q \geqs 3$ we can repeat the argument in the proof of the previous proposition and we deduce that \eqref{e:Q2} holds on substituting $a_1 = 63(q+1)^7$, $b_1 = q^{34}$, $b_2=q^{52}$ and $a_2 = (q+1)^7|W|$, where $W = 2 \times {\rm Sp}_6(2)$ is the Weyl group of $G$. 

Now assume $q=2$ and write $N = T.R$ as before, with $R \cong C_W(w)$ for some $w \in W$ (see Section \ref{ss:tori}). First note that if $N = 3^7.W$ then $N$ is maximal and we can apply Proposition \ref{p:Nmax}. In each of the remaining cases, we claim that $|N| \leqs 729 \times 46080 = a_2$, noting that this bound clearly holds if $N = W$. For $N \ne W$, we see that $|T| \leqs 729$ by inspecting \cite[Tables II, III]{DerFak} and a routine calculation with centralisers in $W$ yields $|R| \leqs 46080$, which justifies the claim. Similarly, we calculate that $\rho_R \leqs 31$ if $w \ne 1$ and $\rho_R = 63$ if $w = 1$, where we recall that $\rho_R$ is the number of reflections in $R$. Then by applying Lemma \ref{l:root}, recalling that we may assume $N = W$ if $w=1$, we see that $N$ contains at most $a_1 = 31.3^7$ long root elements. So if we define $b_1$ and $b_2$ as above, we deduce that \eqref{e:Q2} holds and the result follows.
\end{proof}

Next we turn to the groups $G = E_6^\epsilon(q)$. For $\e=+$, the possibilities for $N = T.R$ can be read off from \cite[Table 1]{GaltStar}. Now assume $\e=-$. Here we recall that there is a bijection from the set of  $F$-classes in $W$ to the set of usual conjugacy classes of $W$, and similarly the $F$-centraliser $R = C_{W,F}(w)$ is isomorphic to $C_{W,F'}(w) \cong C_W(w)$, where $F'$ is the standard Steinberg endomorphism of $\bar{G}$ with fixed point group $E_6(q)$. This allows us to determine the structure of $\bar{T}_w^F$ from the cyclic structure of $\bar{T}_w^{F'}$ given in \cite[Table 1]{GaltStar} by simply substituting $-q$ for $q$ in the order of each cyclic factor, adjusting the sign appropriately. For example, the first three rows in the table correspond to the following possibilities for $N_{\bar{G}^F}(\bar{T}_w)$:  
\[ 
(C_{q+1})^6{:}W, \, ((C_{q+1})^4 \times C_{q^2-1}){:}(S_2 \times S_6), \, ((C_{q+1})^2 \times (C_{q^2-1})^2){:}(D_8 \times S_4).
\]

\begin{prop}\label{p:e6}
If $G = E_6^\epsilon(q)$, then $b(G,N) = 2$.      
\end{prop}

\begin{proof}
Let $x \in G$ be an element of prime order and note that $|x^G|>(q-1)q^{21} = b_1$ if $x$ is a long root element, otherwise $|x^G|>(q-1)q^{31} = b_2$ (see \cite[Proposition 2.11]{BTh_ep}). Now $|N| \leqs (q+1)^6|W| = a_2$, where $W = {\rm PGSp}_4(3)$ is the Weyl group of $G$, and we note that $N$ contains at most $a_1 = 36(q+1)^6$ long root elements by Lemma \ref{l:root}. Putting these estimates together, we deduce that \eqref{e:Q2} holds for $q \geqs 4$. 

Next assume $q=3$. If $G = {}^2E_6(3)$ and $N = 4^6.W$ then $N$ is maximal and we may apply Proposition \ref{p:Nmax}. In each of the remaining cases one can check that $|N| \leqs 2^6 |W| = a_2$ (for example, this follows by inspecting \cite[Table 1]{GaltStar}) and we deduce that \eqref{e:Q2} holds, where $a_1,b_1$ and $b_2$ are defined as before.

For the remainder of the proof, we may assume $q=2$ and $N = T.R$. By inspecting \cite[Table 5.2]{LSS} we first observe that $N$ is maximal when $(\e,N) = (+, 7^3{:}3^{1+2}.{\rm SL}_2(3))$ in which case the result follows from Proposition \ref{p:Nmax}. In considering the remaining cases, let us assume for now that $(\e,N)$ is not one of the following:  
\begin{equation}\label{e:casesE6}
(+,W), \, (+, 3^4.{\rm O}_4^{+}(3)), \, (-, 3^5.W),  \, (-, 3^4.(2 \times {\rm Sp}_4(2))), \,
(-, (3^2 \times 9).(3 \times (S_3 \wr S_2))).
\end{equation}

If we exclude these cases, then one can check that $\rho_R|T| \leqs 810 = a_1$,  with equality if $\epsilon=+$ and $N = (3^2 \times 15).(4 \times S_4)$, where $\rho_R$ is the number of reflections in $R$. In addition, by inspecting \cite[Table 1]{GaltStar} we find that $|N| \leqs 12960 = a_2$. Therefore, Lemma \ref{l:root} implies that the contribution to $\widehat{\mathcal{Q}}(G,N,2)$ from long root elements is at most $a_1^2/b_1$ and one can check that \eqref{e:Q2} holds, where $b_1$ and $b_2$ are defined as above. 

Finally, we handle the cases listed in \eqref{e:casesE6}. In order to determine an upper bound on $\widehat{\mathcal{Q}}(G,N,2)$, it will be convenient to work in $\tilde{N} = N.(2-\epsilon)  = N_{\bar{G}^F}(\bar{T}_w)$. First let $a_1$ be the number of long root elements in $\tilde{N}$ and let $a_2$ be the number of involutions in $\tilde{N}$ that are contained in the $G$-class labelled $A_1^2$ in \cite[Table 22.2.3]{LieS}. Similarly, let $a_3$ be the number of elements $x \in \tilde{N}$ of order $3$ with $C_{\bar{G}}(x)^0 = D_5T_1$ and let $a_4$ be the number of remaining prime order elements in $\tilde{N}$. Then
\begin{equation}\label{e:Q4}
\what{\mathcal{Q}}(G,N,2) < \sum_{i=1}^4 a_i^2/b_i,
\end{equation}
where $b_1$ and $b_2=b_3$ are defined as above and we set $b_4 = 2^{41}$, noting that $|x^G|>b_4$ for all $x \in G$ of prime order other than long root elements, involutions in the $A_1^2$ class and order $3$ elements with a centraliser of type $D_5T_1$. 

To complete the argument, we need to compute $a_1$, $a_2$, $a_3$ and $a_4$. To do this, we proceed as in the proof of \cite[Proposition~2.2]{BTh_comp}, first working with {\sc Magma} to construct $\tilde{N}$ as a subgroup of $\bar{G}^{F^{\ell}} = {\rm Inndiag}(E_6(2^{\ell}))$ for $\ell = 1,2,2,4,6$, respectively (referring to the five cases in \eqref{e:casesE6}). In terms of this embedding, we then construct a set of representatives of the conjugacy classes of elements in $\tilde{N}$ of prime order and we compute the Jordan form of each representative on the adjoint module $V$ for $\bar{G}$. If $x \in \tilde{N}$ is an involution, we can then determine the $G$-class of $x$ by inspecting \cite[Table~6]{Lawunip}. Similarly, if $x$ has order $3$ then we read off $\dim C_V(x) = \dim C_{\bar{G}}(x)$, which allows us to identify the structure of $C_{\bar{G}}(x)^0$. In this way, we deduce that $a_i$ takes the values recorded in Table \ref{tab:ai} and it is straightforward to check that the upper bound in \eqref{e:Q4} yields $\what{\mathcal{Q}}(G,N,2)<1$.
\begin{table}
\[
\begin{array}{llllll} \hline
\e & N & a_1 & a_2 & a_3 & a_4 \\ \hline
+ & W & 36 & 270 & 0 & 6539 \\
+ & 3^4.{\rm O}_4^{+}(3) & 36 & 198 & 0 & 4481 \\  
- & 3^5.W & 108 & 2430 & 54 & 672083 \\
- & 3^4.(2 \times {\rm Sp}_4(2)) & 46 & 450 & 30 & 19571 \\  
- & (3^2 \times 9).(3 \times (S_3 \wr S_2)) & 18 & 81 & 18 & 4262 \\ \hline 
\end{array}
\]
\caption{ }
\label{tab:ai}
\end{table}
\end{proof}

\begin{prop}\label{p:f4}
If $G = F_4(q)$, then $b(G,N) = 2$.      
\end{prop}

\begin{proof}
Let $W = {\rm O}_{4}^{+}(3)$ be the Weyl group of $G$ and note that $|N| \leqs |W|(q+1)^4 = a_1$ and \cite[Proposition 2.11]{BTh_ep} gives $|x^G|>q^{16} = b_1$ for all $x \in G$ of prime order. By Lemma \ref{l:calc}, this implies that 
$\what{\mathcal{Q}}(G,N,2) < a_1^2/b_1$, which is less than $1$ for $q \geqs 7$. 
 
For the remainder of the proof we may assume $q \leqs 5$. We will postpone the analysis of the cases $N = (q+1)^4.W$ with $q \in \{2,3\}$ to the end of the proof. Let $x \in N$ be an element of prime order $r$.

Suppose $q$ is odd. If $r=2$ then $|x^G|> q^{16} = b_1$, whereas $|x^G|> \frac{1}{2}(q-1)q^{21} = b_2$ if $r$ is odd (note that there are no root elements in $N$ by Lemma \ref{l:root}). We claim that $i_2(N) \leqs a_1$ and $|N| \leqs a_2$, where $a_1$ and $a_2$ are defined as in Table \ref{tab:ai2}. Here the upper bound on $|N|$ can be read off from \cite[Table 5.2]{Gager}. To obtain the upper bound on $i_2(N)$, we use {\sc Magma} to construct each possibility for $N$ as a subgroup of $F_4(q^{\ell})$ for some suitable $\ell$, which allows us to compute $i_2(N)$ precisely in each case. It is now straightforward to check that \eqref{e:Q2} holds. 

Now assume $q$ is even. If $r=2$ then $|x^G|> q^{16} = b_1$ and we have $|x^G|> q^{28} = b_2$ if $r$ is odd, or if $r=2$ and $x$ is in the class labelled $A_1\tilde{A}_1$. As above, we note that $|N| \leqs a_2$, where $a_2$ is defined in Table \ref{tab:ai2}. With the aid of {\sc Magma}, we also calculate that there are at most $a_1$  involutions in $N$ that are not in the class labelled $A_1\tilde{A}_1$, where $a_1$ is also given in Table \ref{tab:ai2} (here we can proceed as in the final part of the proof of Proposition \ref{p:e6} in order to determine the $G$-class of each $N$-class of involutions). Once again, one can check that 
\eqref{e:Q2} holds and the result follows.

\begin{table}
\[
\begin{array}{llllll} \hline
q & a_1 & a_2 & b_1 & b_2 \\ \hline
5 & 16000 & 6^4.|W| & 5^{16} & 2.5^{21} \\ 
4 & 6000 & 5^4.|W| &  4^{16} & 3^{28} \\ 
3 & 847 & 12288  & 3^{16} & 3^{21} \\
2 & 234 & 3528  & 2^{16} & 2^{28} \\ \hline
\end{array}
\]
\caption{ }
\label{tab:ai2}
\end{table}

It remains to deal with the two excluded cases. First assume $q=3$ and $N = 4^4.W$, in which case we use {\sc Magma} to construct $N$ as a subgroup of $F_4(9)$. Let $x \in N$ be an element of prime order $r$. If $r=3$ then $|x^G| > 3^{21} = b_3$ since $N$ contains no root elements and we calculate that $i_3(N) = 5120 = a_3$. Similarly, if $r=2$ and $C_{\bar{G}}(x) = B_4$, then $|x^G| > 3^{16} = b_1$ and there are $a_1 = 51$ such elements in $N$. And if $r=2$ and $C_{\bar{G}}(x) \ne B_4$, or if $r \geqs 5$, then $|x^G|>3^{28} = b_2$ and we set $a_2 = |N|$.
Therefore, we conclude that
\[
\what{\mathcal{Q}}(G,N,2) < \sum_{i=1}^3a_i^2/b_i < 1
\] 
and the result follows.

Finally suppose $q=2$ and $N = 3^4.W$. Here we use {\sc Magma} to construct $G$ as a permutation group of degree $139776$ and we find a Sylow $3$-subgroup $H$ of $G$. Now $H$ has a unique abelian subgroup of order $3^4$, which we may assume is the relevant maximal torus $T$. We can now construct $N = N_G(T)$ and then find a random element $x \in G$ with $N \cap N^x = 1$.
\end{proof}

\begin{prop}\label{p:g2}
If $G = G_2(q)'$, then $b(G,N) = 2$.      
\end{prop}

\begin{proof}
Here $|N| \leqs 12(q+1)^2 = a_1$ and we may assume $q \geqs 7$ by Lemma \ref{l:excepcomp}. Let $\mathcal{A}$ be the set of elements $x \in G$ that are either long root elements, or short root elements if $p=3$, or semisimple elements of order $3$ with $C_{\bar{G}}(x) = A_2$. Then if $x \in G$ has prime order, \cite[Proposition 2.11]{BTh_ep} implies that $|x^G|>(q-1)q^{5} = b_2$ if $x \in \mathcal{A}$, otherwise $|x^G| > (q-1)q^7 = b_1$. Set $a_2 = |N \cap \mathcal{A}|$. 

If $q \geqs 16$ then $\what{\mathcal{Q}}(G,N,2) < a_1^2/b_2<1$ and so we may assume $7 \leqs q \leqs 13$. If $q$ is odd then with the aid of {\sc Magma} we compute $a_2 \leqs 2$ (note that $N$ does not contain any long root elements by Lemma \ref{l:root}) and one checks that \eqref{e:Q2} holds. Similarly, if $q=8$ then $a_2 \leqs 29$ and \eqref{e:Q2} holds once again.
\end{proof}

\begin{prop}\label{p:3d4}
If $G = {}^3D_4(q)$, then $b(G,N) = 2$.      
\end{prop}

\begin{proof}\label{p:excepQ}
The possibilities for $T$ and $N$ are recorded in \cite[Table 1.1]{DerMic}. Note that if $x \in G$ has prime order, then either $|x^G|>q^{16} = b_1$, or $x$ is a long root element and $|x^G|>q^{10} = b_2$ (see \cite[Proposition 2.11]{BTh_ep}).

First observe that if $T = C_{q^2\pm q+1} \times C_{q^2\pm q+1}$ or $C_{q^4-q^2+1}$ then $N$ is a maximal subgroup of $G$ and thus $b(G,N) = 2$ by Proposition \ref{p:Nmax}. For the remainder, we may assume $N$ is not one of these possibilities, in which case $|N| \leqs 12(q^3+1)(q+1) = a_1$.

If $q$ is odd then $N$ does not contain long root elements by Lemma \ref{l:root} and thus $\what{\mathcal{Q}}(G,N,2) < a_1^2/b_1 < 1$. Now assume $q$ is even. By Lemma \ref{l:excepcomp}, we may assume $q \geqs 8$. Here 
$\what{\mathcal{Q}}(G,N,2) < a_1^2/b_2$, which is less than $1$ if $q \geqs 16$. Finally, suppose $q=8$ and write $N = T.R$. Here we check that $i_2(R) \leqs 7$, so 
$N$ contains at most $7|T| \leqs 7(q^3+1)(q+1) = a_2$ long root elements by  Lemma \ref{l:root}. It is now routine to check that \eqref{e:Q2} holds, with $a_1, b_1$ and $b_2$ defined as above.
\end{proof}

\begin{prop}\label{p:2f4}
If $G = {}^2F_4(q)'$, then $b(G,N) = 2$.      
\end{prop}

\begin{proof}
By Lemma \ref{l:excepcomp} we may assume $q \geqs 8$. As explained in \cite[Section 7.4]{Gager}, there are $11$ possibilities for $N = T.R$, up to conjugacy, and the structure in each case is recorded in \cite[Table 7.3]{Gager}. Let $x \in G$ be an element of prime order. If $x$ is a long root element, then $|x^G|> (q-1)q^{10} = b_1$, otherwise $|x^G|>(q-1)q^{13} = b_2$ (see \cite[Proposition 2.11]{BTh_ep}). Now $|T| \leqs (\sqrt{q} + 1)^4$ (see Lemma \ref{l:Tsize}), $|R| \leqs 96$ and by applying Lemma \ref{l:root} we see that $N$ contains at most $a_1 = 24(\sqrt{q} + 1)^4$ long root elements since the Weyl group of type $F_4$ contains $24$ reflections. One can now check that \eqref{e:Q2} holds, with $a_2 = 96(\sqrt{q} + 1)^4$.
\end{proof}

\begin{prop}\label{p:2g2}
If $G ={}^2G_2(q)'$, then $b(G,N) = 2$.      
\end{prop}

\begin{proof}
Recall that we may assume $q \geqs 27$. The four possibilities for $N = T.R$ (up to conjugacy) are given in \cite[Proposition~7.4]{Gager}. By \cite[Theorem C]{K88}, either $N$ is maximal and hence $b(G,N) = 2$ by Proposition \ref{p:Nmax}, or $N = C_{q-1}.2$. To handle the latter case, set $a = |N|$ and note that $|x^G|> (q-1)q^{3} = b$ for any prime order element $x \in N$ by \cite[Proposition 2.11]{BTh_ep}. By Lemma \ref{l:calc}, this implies that $\what{\mathcal{Q}}(G,N,2) < a^2/b$, which is less than $1$ for $q \geqs 27$. 
\end{proof}

\vs

This completes the proof of Theorem \ref{t:main} when $G$ is a simple exceptional group of Lie type. By combining this with the main results on classical groups in Sections \ref{s:lu}-\ref{s:o_even}, we conclude that the proof of Theorem \ref{t:main} is complete. Finally, we establish Corollary \ref{c:main}.

\begin{proof}[Proof of Corollary \ref{c:main}]
Let $G = O^{p'}(\bar{G}^F)$ be a finite simple group of Lie type over a field of characteristic $p$ and let $N = N_G(\bar{T})$, where $\bar{T}$ is an $F$-stable maximal torus of $\bar{G}$. If $b(G,N) = 2$ then $N\cap N^x = 1$ for some $x \in G$, so by Theorem \ref{t:main} we may assume $(G,N)$ is one of the cases listed in Table \ref{tab:main}. The result now follows by combining the statements of Theorems \ref{t:lu}, \ref{t:sp} and \ref{t:so_even} with Propositions \ref{l:psl2}--\ref{l:psl5}.
\end{proof}

\end{document}